\documentclass[12pt]{amsart}

\usepackage{graphicx}
\usepackage{amsmath,amsxtra,amssymb,latexsym, amscd,amsthm}

\usepackage{cite}
\theoremstyle{plain}
\newtheorem{thm}{Theorem}[section]

\newtheorem{lem}[thm]{Lemma}

\newtheorem{claim}[thm]{Claim}
\theoremstyle{definition}

\theoremstyle{remark}
\newtheorem{rmk}{Remark}

\numberwithin{equation}{section}
\numberwithin{figure}{section}
\numberwithin{table}{section}

\newcommand{\wt}{\operatorname{wt}}

\newcommand{\M}{\operatorname{M}}

\newcommand{\E}{\operatorname{E}}
\newcommand{\Od}{\operatorname{O}}

\newcommand{\Ob}{\operatorname{Ob}}

\usepackage{geometry}
\begin{document}

\title[Enumeration of Hybrid Domino-Lozenge Tilings III]{Enumeration of Hybrid Domino-Lozenge Tilings III:\\ Centrally Symmetric Tilings}

\author{Tri Lai}
\address{Department of Mathematics, University of Nebraska -- Lincoln, Lincoln, NE 68588, USA}
\email{tlai3@unl.edu}
\thanks{T.L. was supported in part  by Simons Foundation Collaboration Grant (\# 585923).}



\begin{abstract}
We use the subgraph replacement method to investigate new properties of the tilings of regions on the square lattice with diagonals drawn in. In particular, we show that the centrally symmetric tilings of a generalization of the Aztec diamond are always enumerated by a simple product formula. This result generalizes the previous work of Ciucu (1997) and Yang (1992) about symmetric tilings of the Aztec diamond. We also use our method to prove a closed form product formula for the number of centrally symmetric tilings of a quasi-hexagon.
\end{abstract}

\maketitle


\section{Introduction}\label{Intro}

The hybrid domino-lozenge tilings were first studied by J. Propp in the 1990s (see \cite{Propp} and the list of references therein). In 1996, C. Douglas  \cite{Doug} proved a conjecture posed by J. Propp about the number of tilings of  an analog of the Aztec diamond on the square lattice with every second diagonal\footnote{From now on,  we use the word ``\emph{diagonal}" to mean ``\emph{southwest-to-northeast diagonal}''} drawn in (see Figure \ref{Douglas}  for several first regions of Douglas and Figure \ref{sampletiling}(a) for a sample tiling). In particular, Douglas showed that the region of order $n$ has exactly $2^{2n(n+1)}$ tilings.

\begin{figure}\centering
\begin{picture}(0,0)%
\includegraphics{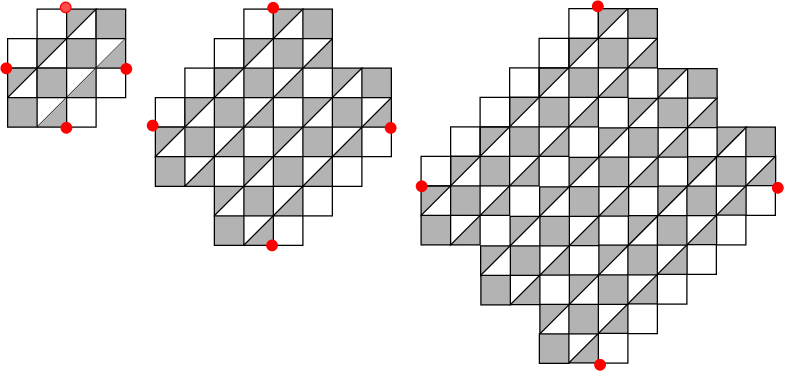}%
\end{picture}%
\setlength{\unitlength}{3947sp}%
\begingroup\makeatletter\ifx\SetFigFont\undefined%
\gdef\SetFigFont#1#2#3#4#5{%
  \reset@font\fontsize{#1}{#2pt}%
  \fontfamily{#3}\fontseries{#4}\fontshape{#5}%
  \selectfont}%
\fi\endgroup%
\begin{picture}(6283,2966)(402,-2292)
\put(2363,-1641){\makebox(0,0)[lb]{\smash{{\SetFigFont{12}{14.4}{\familydefault}{\mddefault}{\updefault}{$n=2$}%
}}}}
\put(6260,-1878){\makebox(0,0)[lb]{\smash{{\SetFigFont{12}{14.4}{\familydefault}{\mddefault}{\updefault}{$n=3$}%
}}}}
\put(828,-696){\makebox(0,0)[lb]{\smash{{\SetFigFont{12}{14.4}{\familydefault}{\mddefault}{\updefault}{$n=1$}%
}}}}
\end{picture}%
\caption{Several initial regions in  Douglas' Theorem \cite{Doug}: the regions of order $n=1$, $n=2$ and $n=3$. The figure was first introduced in \cite{Tri1}.}
\label{Douglas}
\end{figure}
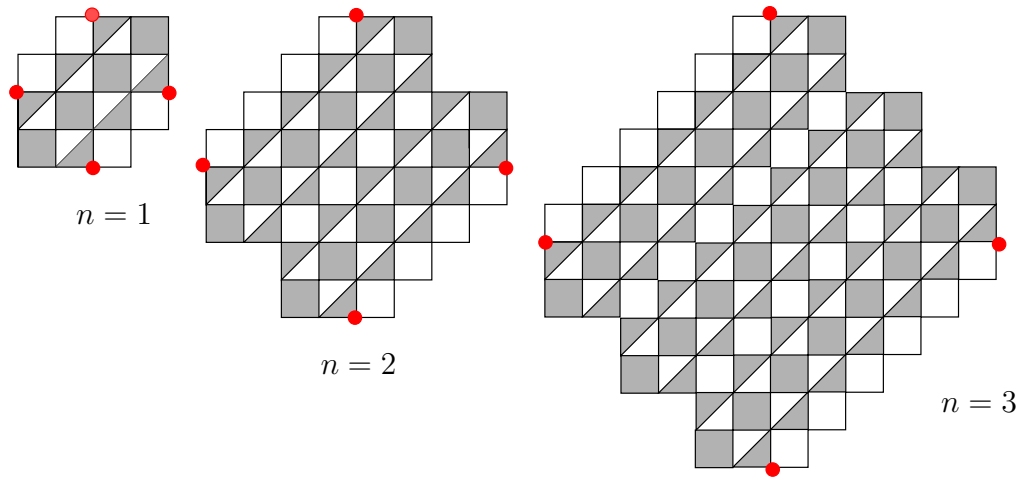

\begin{figure}\centering
\includegraphics[width=7cm]{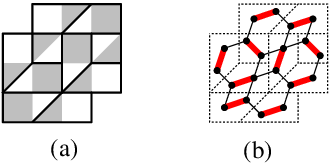}
\caption{ (a)  A tiling of the left region in Figure \ref{Douglas} and (b) the corresponding perfect matching of its dual graph.}\label{sampletiling}
\end{figure}

Recently, the author \cite{Tri6,Tri7} generalized Douglas' theorem and the Aztec diamond theorem of N. Elkies, G. Kuperberg, M. Larsen and J. Propp \cite{Elkies1,Elkies2} by enumerating tilings of a family of 4-sided regions on the square lattice with arbitrary diagonals drawn in (see Figure \ref{Douglasgen} for an example of the region). We call this region a \emph{Douglas region}\footnote{The region was called a \emph{generalized Douglas region} in \cite{Tri6,Tri7}.} (the detailed definition of the region will be given in the next section).  In particular, we showed that the tiling number of a Douglas region is always given by a power of $2$ (see Theorem 4 in \cite{Tri6}). This implies Douglas' theorem when the distances between any two consecutive drawn-in diagonals are $2\sqrt{2}$, and the Aztec diamond theorem when there is no drawn-in diagonal.

\begin{figure}\centering
\setlength{\unitlength}{3947sp}%
\begingroup\makeatletter\ifx\SetFigFont\undefined%
\gdef\SetFigFont#1#2#3#4#5{%
  \reset@font\fontsize{#1}{#2pt}%
  \fontfamily{#3}\fontseries{#4}\fontshape{#5}%
  \selectfont}%
\fi\endgroup%
\resizebox{10cm}{!}{
\begin{picture}(0,0)%
\includegraphics{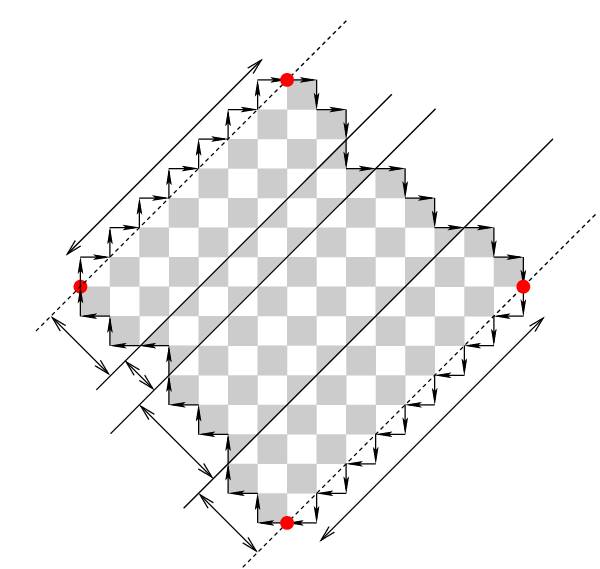}%
\end{picture}%

\begin{picture}(4794,4535)(924,-4323)
\put(5446,-1591){\makebox(0,0)[lb]{\smash{{\SetFigFont{12}{14.4}{\rmdefault}{\mddefault}{\itdefault}{$\ell'$}%
}}}}
\put(1800,-4216){\makebox(0,0)[lb]{\smash{{\SetFigFont{12}{14.4}{\rmdefault}{\mddefault}{\itdefault}{$\textbf{d}_4\frac{\sqrt{2}}{2}=\textbf{4}\frac{\sqrt{2}}{2}$}%
}}}}
\put(3136,-196){\makebox(0,0)[lb]{\smash{{\SetFigFont{12}{14.4}{\rmdefault}{\mddefault}{\itdefault}{$A$}%
}}}}
\put(3496, 36){\makebox(0,0)[lb]{\smash{{\SetFigFont{12}{14.4}{\rmdefault}{\mddefault}{\itdefault}{$\ell$}%
}}}}
\put(5347,-2182){\makebox(0,0)[lb]{\smash{{\SetFigFont{12}{14.4}{\rmdefault}{\mddefault}{\itdefault}{$B$}%
}}}}
\put(3221,-4308){\makebox(0,0)[lb]{\smash{{\SetFigFont{12}{14.4}{\rmdefault}{\mddefault}{\itdefault}{$C$}%
}}}}
\put(1704,-939){\makebox(0,0)[lb]{\smash{{\SetFigFont{12}{14.4}{\rmdefault}{\mddefault}{\itdefault}{$a=7$}%
}}}}
\put(4479,-3474){\makebox(0,0)[lb]{\smash{{\SetFigFont{12}{14.4}{\rmdefault}{\mddefault}{\itdefault}{$w=8$}%
}}}}
\put(600,-2821){\makebox(0,0)[lb]{\smash{{\SetFigFont{12}{14.4}{\rmdefault}{\mddefault}{\itdefault}{$\textbf{d}_1\frac{\sqrt{2}}{2}=\textbf{4}\frac{\sqrt{2}}{2}$}%
}}}}
\put(850,-3234){\makebox(0,0)[lb]{\smash{{\SetFigFont{12}{14.4}{\rmdefault}{\mddefault}{\itdefault}{$\textbf{d}_2\frac{\sqrt{2}}{2}=\textbf{2}\frac{\sqrt{2}}{2}$}%
}}}}
\put(1450,-3646){\makebox(0,0)[lb]{\smash{{\SetFigFont{12}{14.4}{\rmdefault}{\mddefault}{\itdefault}{$\textbf{d}_3\frac{\sqrt{2}}{2}=\textbf{5}\frac{\sqrt{2}}{2}$}%
}}}}
\put(1201,-2124){\makebox(0,0)[lb]{\smash{{\SetFigFont{12}{14.4}{\rmdefault}{\mddefault}{\itdefault}{$D$}%
}}}}
\end{picture}}
\caption{The Douglas region $\mathcal{D}_{7}(4,2,5,4)$. The figure was first introduced in \cite{Tri6}.}
\label{Douglasgen}
\end{figure}

Propp \cite{Propp} also investigated a `natural hybrid' between the Aztec diamond and a lozenge hexagon on the square lattice with every third diagonal drawn in, called a \emph{quasi-hexagon} and defined in detail in Section 2. Finding an explicit tiling formula for a quasi-hexagon was a long-standing open problem in the field (see Problem 16 on Propp's well-known list of 32 open problems in enumeration of tilings \cite{Propp}). The author \cite{Tri1} solved this problem by using the subgraph replacement method. In general, there is no simple product formula for the number of tilings of a quasi-hexagon. However, in the symmetric case, we have a simple product formula, which is a certain product of a power of $2$ and an instance of MacMahon's tiling formula (\ref{Maceq}) for a semi-regular hexagon on the triangular lattice \cite{McM}. The author \cite{Tri2} also enumerated tilings of an $8$-vertex counterpart of the quasi-hexagons, called \emph{quasi-octagons}.

Inspired by the work of B.-Y. Yang \cite{Yang} and M. Ciucu \cite{Ciucu} about the symmetric tilings of the Aztec diamond, we consider the  \emph{centrally symmetric tilings} (i.e. the tilings which are invariant under $180^{\circ}$ rotations) of a Douglas region. We actually investigate a more general case when certain portions of the region have been removed along a symmetry axis as in Figure \ref{Douglashole} (the black parts indicate the removed portions). We call this removed portions \emph{holes}. We show that the number of centrally symmetric tilings of such a Douglas region with holes is always given by a closed form product formula (see Theorem \ref{symthm1} in Section 2). See Figure \ref{Douglashole} for Douglas regions with holes and Figure \ref{sampletiling2}(a) for a centrally symmetric tiling of a Douglas region with holes.

The study of symmetric (lozenge) tilings of a hexagon on the triangular lattice  dates back to the late 1890s when MacMahon conjectured the $q$-enumeration of the symmetric plane partitions \cite{McM2}. About one hundred years later, all 10 symmetry classes of plane partitions were collected in Stanley's classical paper \cite{Stanley2}. Each of these symmetry classes can be translated into a certain class of symmetric lozenge tilings of a hexagon. As one of the 10 symmetry classes,  the \emph{self-complementary plane partitions} correspond to the centrally symmetric tilings of a hexagon. Stanley \cite{Stanley2}  showed that the number of self-complementary plane partitions, and hence the number of centrally symmetric tilings of a hexagon,  is always given by a simple product formula. Viewing a quasi-hexagon as a generalization of a lozenge hexagon,  we now investigate centrally symmetric tilings of  quasi-hexagons. In particular,  we use the subgraph replacement method to show that the number of centrally symmetric tilings of a quasi-hexagon is also given by a simple product formula (see Theorem \ref{hexsym} in Section 2). 

The rest of this paper is organized as follows. We give detailed definitions of the Douglas regions and the quasi-hexagons, and the statements of our main results (Theorems \ref{symthm1} and \ref{hexsym}) in Section \ref{Statement}. Section \ref{Background} is devoted to several fundamental results in the subgraph replacement method that will be employed in our proofs. In Section \ref{Mainproof1}, we enumerate the perfect matchings  of an Aztec rectangle graph with holes, that itself  can be considered as a generalization of the related work of  B.Y. Yang \cite{Yang} and M. Ciucu \cite{Ciucu} in the case of the Aztec diamonds. We will use this enumeration in our  proof of Theorem \ref{symthm1} in Section \ref{Mainproof2}. Finally, in Section \ref{Mainproof3},  we present the proof of Theorem \ref{hexsym}.

\section{Statement of the main results}\label{Statement}

A lattice divides the plane into disjoint fundamental regions, called \emph{cells}. A (lattice) \textit{region} is a finite connected union of cells. A \textit{tile} is the union of any two cells sharing an edge. A \textit{tiling} of a region $\mathcal{R}$ is a covering of $\mathcal{R}$ by tiles, such that there are no gaps or overlaps. The number of tilings of the region $\mathcal{R}$ is denoted by $\M(\mathcal{R})$.

Let $\ell$ be a fixed drawn-in diagonal on the square lattice. Assume that $k$ more diagonals have been drawn in above $\ell$ with the distances between two consecutive ones from the top $d_1\frac{\sqrt{2}}{2}$, $d_2\frac{\sqrt{2}}{2}$,\dots,$d_k\frac{\sqrt{2}}{2}$, and $k'$ more diagonals have been drawn in below $\ell$ with the distances between two consecutive ones from the bottom $d'_1\frac{\sqrt{2}}{2}$, $d'_2\frac{\sqrt{2}}{2}$,\dots,$d'_{k'}\frac{\sqrt{2}}{2}$ (see Figure \ref{hexagonregion}). Next, we color black and white the dissection obtained from the above set-up of drawn-in diagonals on the square lattice, so that two cells sharing an edge have different colors.

We define the \emph{quasi-hexagon} $\mathcal{H}_{a}(d_1,d_2,\dots,d_k; d'_1,d'_2,\dots,d'_{k'})$ as follows. Pick a lattice point $A$ on the the top drawn-in diagonal. Starting from $A$, we go south or east in each step so that the black cell stays on the left. The resulting lattice path from $A$ intersects the diagonal $\ell$ at a lattice point $B$. From $B$, we go south or east so that the white cell stays on the left in each step. Our lattice path stops when reaching the bottom drawn-in diagonal at a lattice point $C$. The described lattice path passing $A$, $B$ and $C$ is the southwestern boundary of the region. Next, we pick a lattice point $F$ on the top drawn-in diagonal such that $F$ is $a\sqrt{2}$ units to the right of $A$. The northeastern boundary is obtained from the southwestern one by reflecting about the perpendicular bisector of the segment $AF$. Assume that the northeastern boundary intersects $\ell$ and the bottom drawn-in diagonal at $E$ and $D$, respectively. We complete the boundary of the region by connecting $C$ and $D$, and $F$ and $A$ along the corresponding drawn-in diagonals. The six lattice points $A,B,C,D,E,$ and $F$ are called the \emph{vertices} of the region, and the diagonal $\ell$ is called the (southwest-to-northeast) \emph{axis} of the region.

\begin{figure}\centering
\begin{picture}(0,0)%
\includegraphics{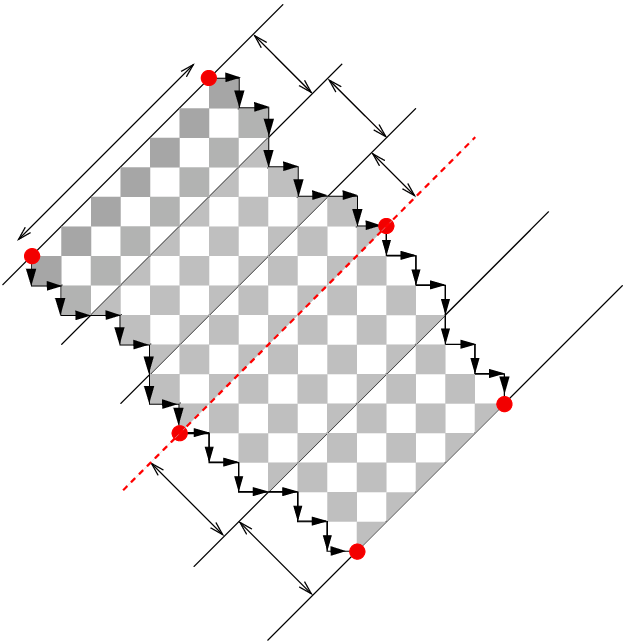}%
\end{picture}
\setlength{\unitlength}{3947sp}%
\begingroup\makeatletter\ifx\SetFigFont\undefined%
\gdef\SetFigFont#1#2#3#4#5{%
  \reset@font\fontsize{#1}{#2pt}%
  \fontfamily{#3}\fontseries{#4}\fontshape{#5}%
  \selectfont}%
\fi\endgroup%
\begin{picture}(5276,5190)(406,-4616)
\put(2200,-4276){\makebox(0,0)[lb]{\smash{{\SetFigFont{12}{14.4}{\familydefault}{\mddefault}{\updefault}{$d'_1\frac{\sqrt{2}}{2}$}%
}}}}
\put(3200,-229){\makebox(0,0)[lb]{\smash{{\SetFigFont{12}{14.4}{\familydefault}{\mddefault}{\updefault}{$d_2\frac{\sqrt{2}}{2}$}%
}}}}
\put(3500,-793){\makebox(0,0)[lb]{\smash{{\SetFigFont{12}{14.4}{\familydefault}{\mddefault}{\updefault}{$d_3\frac{\sqrt{2}}{2}$}%
}}}}
\put(2600,158){\makebox(0,0)[lb]{\smash{{\SetFigFont{12}{14.4}{\familydefault}{\mddefault}{\updefault}{$d_1\frac{\sqrt{2}}{2}$}%
}}}}
\put(4400,-646){\makebox(0,0)[lb]{\smash{{\SetFigFont{12}{14.4}{\familydefault}{\mddefault}{\updefault}{$\ell$}%
}}}}
\put(800,-586){\makebox(0,0)[lb]{\smash{{\SetFigFont{12}{14.4}{\familydefault}{\mddefault}{\updefault}{$a\sqrt{2}$}%
}}}}
\put(1400,-3778){\makebox(0,0)[lb]{\smash{{\SetFigFont{12}{14.4}{\familydefault}{\mddefault}{\updefault}{$d'_2\frac{\sqrt{2}}{2}$}%
}}}}
\put(0,-1704){\makebox(0,0)[lb]{\smash{{\SetFigFont{12}{14.4}{\familydefault}{\mddefault}{\updefault}{$A$}%
}}}}
\put(1400,-3076){\makebox(0,0)[lb]{\smash{{\SetFigFont{12}{14.4}{\familydefault}{\mddefault}{\updefault}{$B$}%
}}}}
\put(3400,-4141){\makebox(0,0)[lb]{\smash{{\SetFigFont{12}{14.4}{\familydefault}{\mddefault}{\updefault}{$C$}%
}}}}
\put(4600,-2986){\makebox(0,0)[lb]{\smash{{\SetFigFont{12}{14.4}{\familydefault}{\mddefault}{\updefault}{$D$}%
}}}}
\put(2000,344){\makebox(0,0)[lb]{\smash{{\SetFigFont{12}{14.4}{\familydefault}{\mddefault}{\updefault}{$F$}%
}}}}
\put(3700,-1329){\makebox(0,0)[lb]{\smash{{\SetFigFont{12}{14.4}{\familydefault}{\mddefault}{\updefault}{$E$}%
}}}}
\end{picture}
\caption{The quasi-hexagon $\mathcal{H}_{6}(4,4,3;\  5,5)$. The figure first appeared in \cite{Tri1}.}
\label{hexagonregion}
\end{figure}
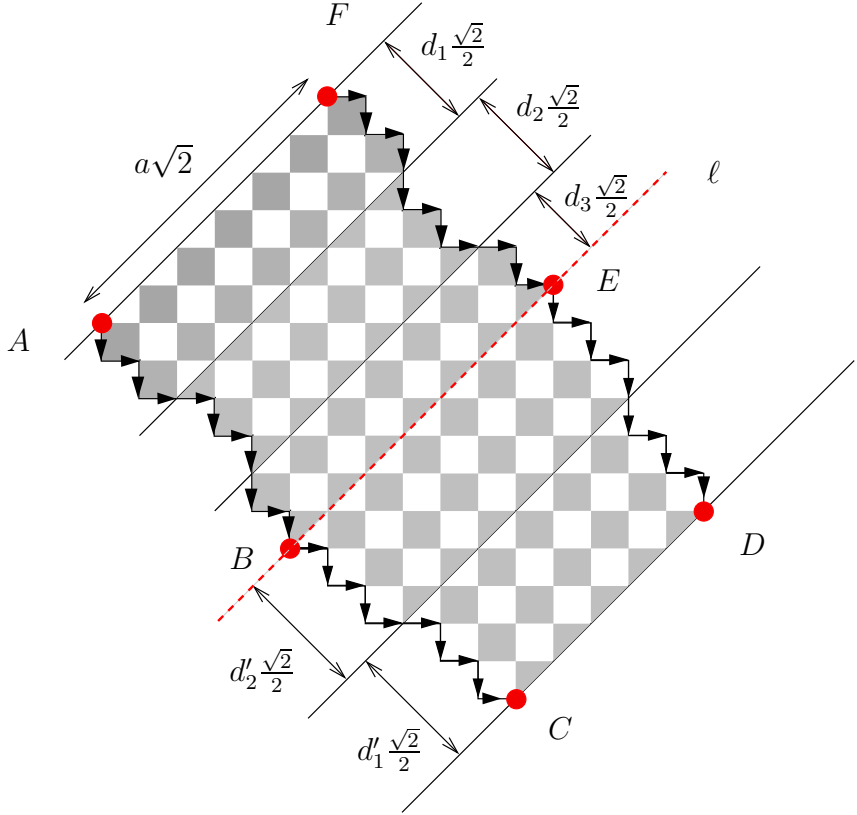

The cells in a quasi-hexagon are unit squares or triangles. The triangular cells only appear along the drawn-in diagonals. A \textit{row of cells} consists of all the triangular cells of a given color with bases resting on a fixed lattice diagonal, or  consists of all the square cells\footnote{From now on, we use the words ``\emph{triangle(s)}" and ``\emph{square(s)}" to mean ``\emph{triangular cell(s)}" and ``\emph{square cell(s)}", respectively.} (of a given color) passed through by a fixed lattice diagonal.

Define the \emph{Douglas region}  $\mathcal{D}_{a}(d_1,\dots,d_k)$ to be the region obtained from the portion of the region  $\mathcal{H}_{a}(d_1,\dots,d_k; d'_1,\dots,d'_{k'})$ above the axis $\ell$ by replacing the triangles running along the top and the bottom by squares of the same color (see Figure \ref{Douglasgen}). The Douglas region was first investigated in \cite{Tri1}, and also in \cite{Tri6} and \cite{Tri7}, as a common generalization of Douglas'  original regions \cite{Doug} and the Aztec diamonds \cite{Elkies1,Elkies2}.

\begin{rmk}
As mentioned in \cite{Tri1} (Theorem 2.1(a) and Theorem 2.3(a)), if the triangles running along the bottom of a  quasi-hexagon or a Douglas region are black, then the region has no tilings. \textit{Therefore, from now on, we assume that the bottom triangles are white.} This is equivalent to the fact that the last step of the southwestern boundary is an east step.
\end{rmk}

For any finite set of integers  $A=\{a_1,a_2,\dotsc,a_n\}$, $n\geq 0$, we define four functions
\begin{equation}\label{E}
\E(a_1,a_2,\dotsc,a_n)=\frac{2^{n^2}}{0!2!4!\dotsc(2n-2)!}\prod_{1\leq i<j\leq n}(a_j-a_i)\prod_{1\leq i<j\leq n}(a_i+a_j-1),
\end{equation}
\begin{equation}\label{O}
\Od(a_1,a_2,\dotsc,a_n)=\frac{2^{n^2}}{1!3!5!\dotsc(2n-1)!}\prod_{1\leq i<j \leq n}(a_j-a_i)\prod_{1\leq i\leq j \leq n}(a_i+a_j-1),
\end{equation}
\begin{equation}\label{Eb}
\overline{\E}(a_1,a_2,\dotsc,a_n)=\frac{2^{n^2}\prod_{i=1}^{n}a_i}{0!2!4!\dotsc(2n-2)!}\prod_{1\leq i<j\leq n}(a_j-a_i)\prod_{1\leq i\leq j\leq n}(a_i+a_j),
\end{equation}
and
\begin{equation}\label{Ob}
\overline{\Od}(a_1,a_2,\dotsc,a_n)=\frac{2^{n^2}\prod_{i=1}^{n}a_i}{1!3!5!\dotsc(2n-1)!}\prod_{1\leq i<j\leq n}(a_j-a_i)\prod_{1\leq i< j\leq n}(a_i+a_j),
\end{equation}
where the empty products are equal to 1 by convention. The functions $\E$ and $\Od$ were introduced by Jockusch and Propp in \cite{JP} as the number of the so-called \emph{anti-symmetric monotone triangles}, and the functions $\overline{\E}$ and $\overline{\Od}$ were introduced  by the author in \cite{Tri3} as the tiling numbers of a family of regions called \emph{quartered Aztec rectangles}.

Consider a Douglas region $\mathcal{D}:=\mathcal{D}_{a}(\textbf{d})=\mathcal{D}_{a}(d_1,d_2,\dotsc,d_k)$ that  admits the southwest-to-northeast symmetry axis $\alpha$. It is easy to see that we must have (1) $d_{i}=d_{k-i+1}$,  (2) $k$ is odd, and (3) $\alpha$ is \textit{not} a drawn-in diagonal (i.e., all the cells running along $\alpha$ are squares). We label the squares passed by $\alpha$ as follows. If the symmetry center of $\mathcal{D}$ stays inside one of these squares, we call this square the \textit{central cell}, and label it by $0$. Next, we label the two squares closest to the center by $1$, we label the two squares that are second closest to the center by $2$, and so on (see Figure \ref{Douglashole}). A cell of $\mathcal{D}$ is said to be \emph{regular} if it is either a black square or a black triangle pointing away from $\alpha$. We define the \emph{height} $h(\mathcal{D})$ of $\mathcal{D}$ to be the number of rows of regular cells above $\alpha$ or passed by $\alpha$. By the symmetry, $h(D)$ is also the number of rows of regular cells below $\alpha$ or passed by $\alpha$. The number of regular cells which is on or above $\alpha$ is denoted by $\mathcal{C}(\mathcal{D})$, and we usually call it the \emph{number of upper regular cells}.  The number $w(\mathcal{D})$ of squares passed through by $\alpha$ is called the \emph{width} of $\mathcal{D}$. We call a row of an odd number of black triangles pointing toward $\alpha$ and above $\alpha$ a \textit{singular row}. The number of singular rows $\tau(\mathcal{D})$ is called the \textit{defect} of $\mathcal{D}$. For example the left region in Figure  \ref{Douglashole} has respectively the height, the number of upper regular cells, the width, and the defect $5,63,12,0$; the region on the right of the figure has these parameters $5,41,7,1$, respectively.

We remove all squares having labels in a subset $\mathcal{S}$ of $\{0,1,2,\dotsc,\lfloor \frac{w(\mathcal{D})}{2}\rfloor\}$ along $\alpha$. Denote by the resulting region by $\mathcal{D}_{a}(\textbf{d};\mathcal{S})$ (see Figure \ref{Douglashole} for examples; the black squares indicate the ones that have been removed). 
\begin{figure}\centering
\setlength{\unitlength}{3947sp}%
\begingroup\makeatletter\ifx\SetFigFont\undefined%
\gdef\SetFigFont#1#2#3#4#5{%
  \reset@font\fontsize{#1}{#2pt}%
  \fontfamily{#3}\fontseries{#4}\fontshape{#5}%
  \selectfont}%
\fi\endgroup%
\resizebox{14cm}{!}{
\begin{picture}(0,0)%

\includegraphics{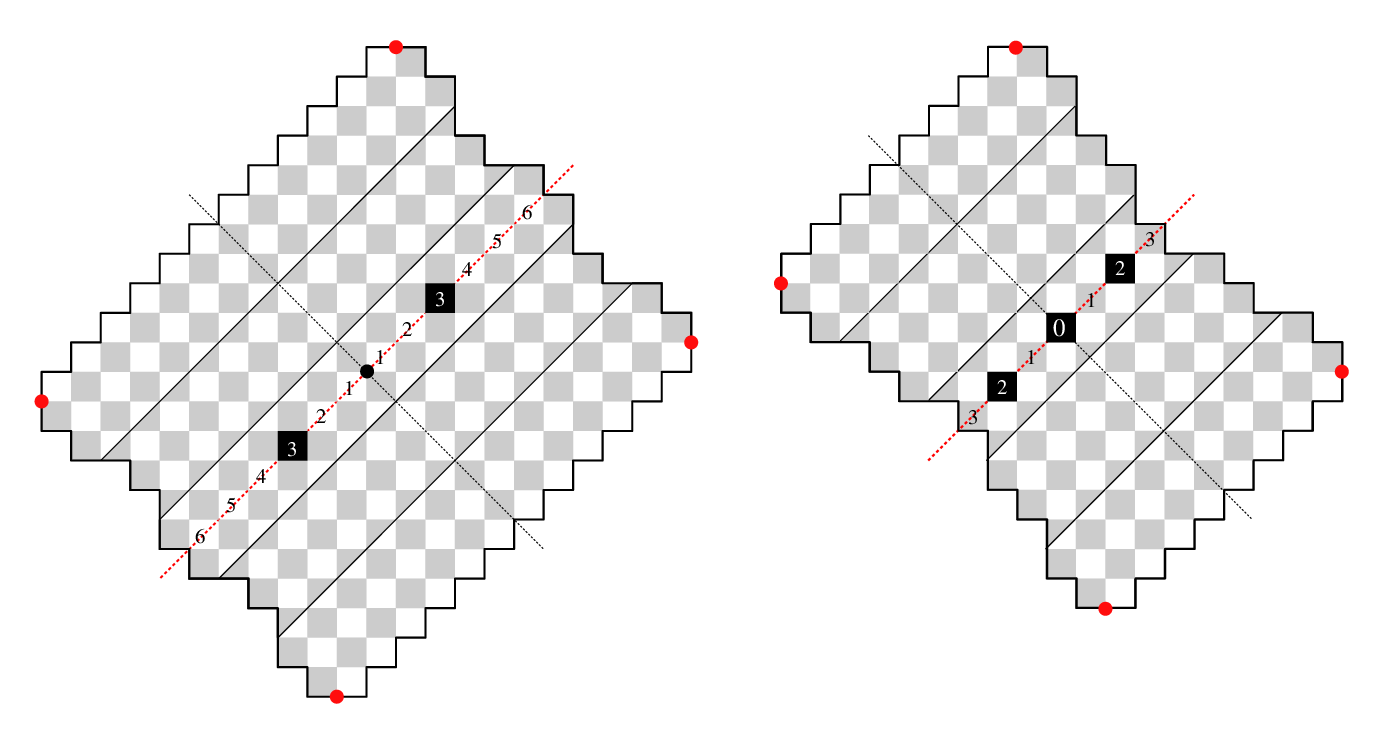}%
\end{picture}%
%

\begin{picture}(11085,5873)(849,-5373)
\put(5552,-874){\makebox(0,0)[lb]{\smash{{\SetFigFont{10}{12.0}{\rmdefault}{\mddefault}{\updefault}{$\alpha$}%
}}}}
\put(10630,-866){\makebox(0,0)[lb]{\smash{{\SetFigFont{10}{12.0}{\rmdefault}{\mddefault}{\updefault}{$\alpha$}%
}}}}
\put(3469,-5358){\makebox(0,0)[lb]{\smash{{\SetFigFont{12}{14.4}{\rmdefault}{\mddefault}{\updefault}{$C$}%
}}}}
\put(3960,281){\makebox(0,0)[lb]{\smash{{\SetFigFont{12}{14.4}{\rmdefault}{\mddefault}{\updefault}{$A$}%
}}}}
\put(8919,277){\makebox(0,0)[lb]{\smash{{\SetFigFont{12}{14.4}{\rmdefault}{\mddefault}{\updefault}{$A$}%
}}}}
\put(11792,-2542){\makebox(0,0)[lb]{\smash{{\SetFigFont{12}{14.4}{\rmdefault}{\mddefault}{\updefault}{$B$}%
}}}}
\put(864,-2776){\makebox(0,0)[lb]{\smash{{\SetFigFont{12}{14.4}{\rmdefault}{\mddefault}{\updefault}{$D$}%
}}}}
\put(9618,-4655){\makebox(0,0)[lb]{\smash{{\SetFigFont{12}{14.4}{\rmdefault}{\mddefault}{\updefault}{$C$}%
}}}}
\put(6779,-1832){\makebox(0,0)[lb]{\smash{{\SetFigFont{12}{14.4}{\rmdefault}{\mddefault}{\updefault}{$D$}%
}}}}
\put(6587,-2308){\makebox(0,0)[lb]{\smash{{\SetFigFont{12}{14.4}{\rmdefault}{\mddefault}{\updefault}{$B$}%
}}}}
\end{picture}}
\caption{The Douglas regions with holes: $\mathcal{D}_{12}(4,4,4,4,4; \{3\})$ (left) and $\mathcal{D}_{8}(4,5,4,5,4; \{2\})$ (right).}\label{Douglashole}
\end{figure}
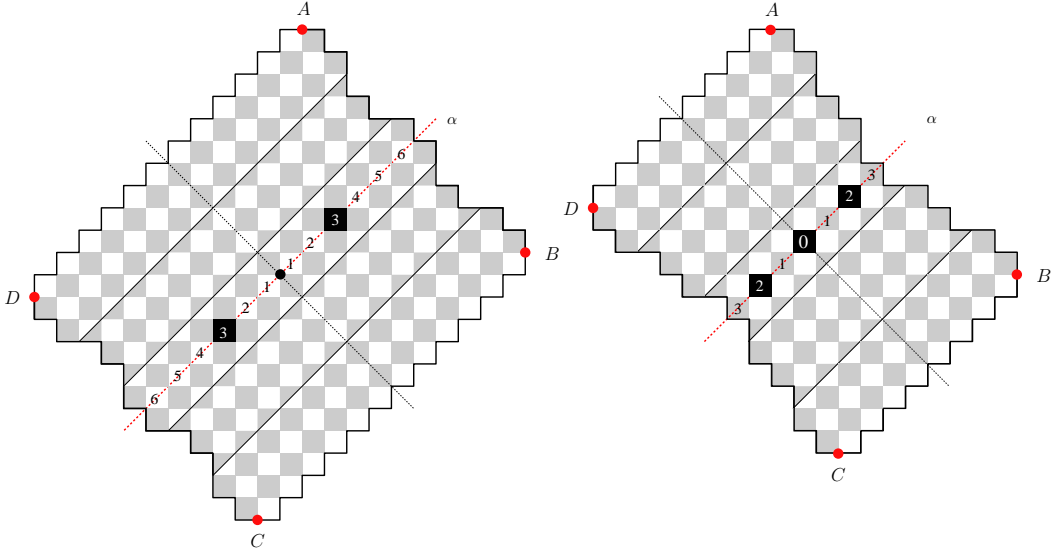

\begin{figure}\centering
\includegraphics[width=12cm]{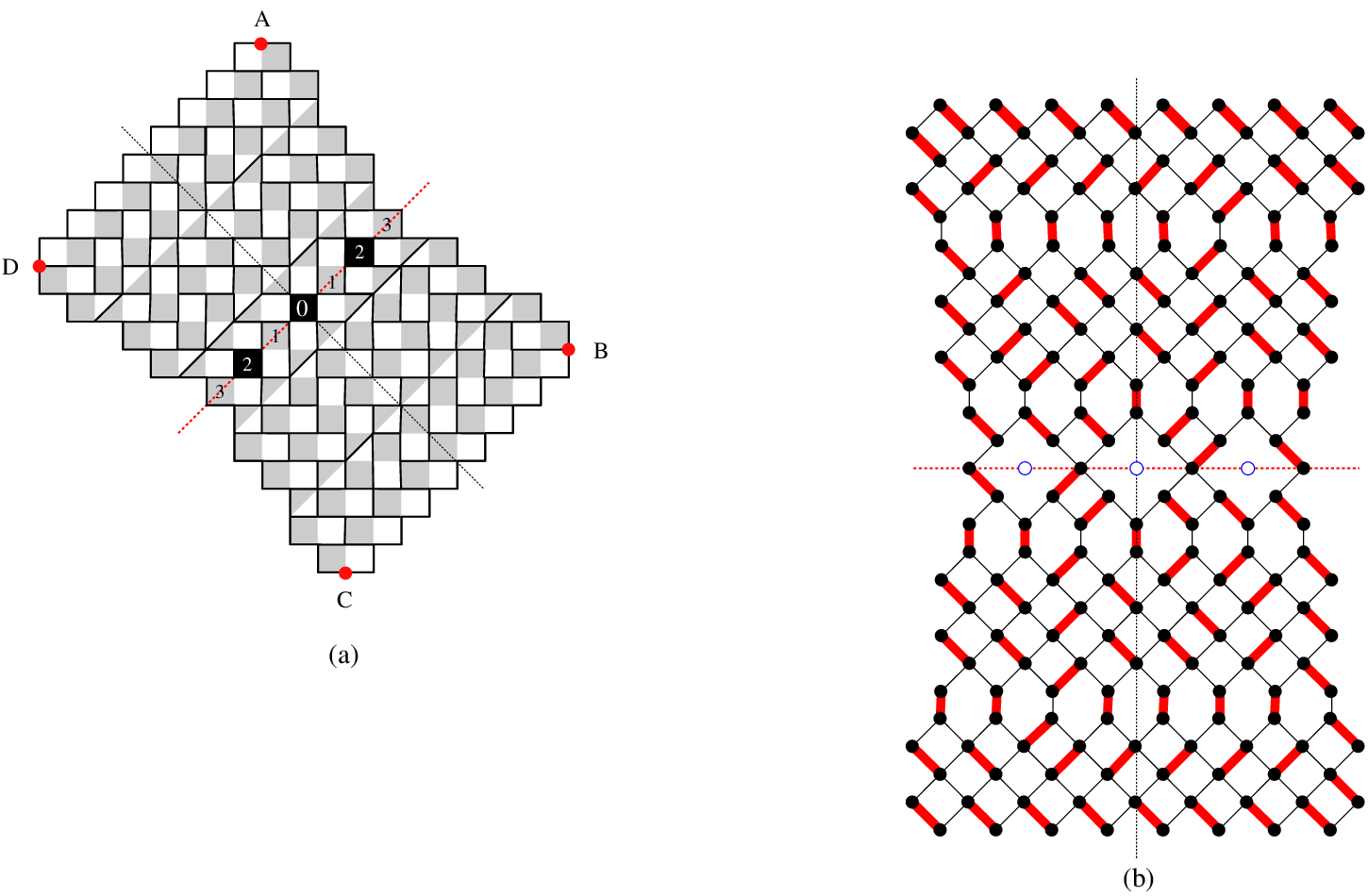}
\caption{(a) A centrally symmetric tiling of the region $\mathcal{D}_{8}(4,5,4,5,4; \{2\})$ in Figure \ref{Douglashole} (right).  (b) A centrally symmetric perfect matching of the dual graph of $\mathcal{D}_{8}(4,5,4,5,4; \{2\})$ (rotated $45^{\circ}$).}\label{sampletiling2}
\end{figure}

We notice that if $\mathcal{D}_{a}(\textbf{d};\mathcal{S})$ admits a tiling, then the number of squares removed equals  $|w(\mathcal{D})-2h(\mathcal{D})|$ if $\alpha$ passes white squares, and equals $|w(\mathcal{D})-2h(\mathcal{D})+1|$, otherwise. Moreover, in the latter case, the central cell must be removed. The number of centrally symmetric tilings of the region  $\mathcal{D}_{a}(\textbf{d};\mathcal{S})$ is given by the following theorem. In this paper, we use the notation $\M^*(\mathcal{R})$ for the number of centrally symmetric tilings of $\mathcal{R}$.

\begin{thm}\label{symthm1}
Consider a positive integer $a$ and a sequence of positive integers $\textbf{d}=(d_i)_{i=1}^{k}$ so that the Douglas region $\mathcal{D}:=\mathcal{D}_{a}(\textbf{d})$ admits a southwest-to-northeast symmetry axis $\alpha$ and has width $w=w(\mathcal{D})$, height $h=h(\mathcal{D})$, defect $\tau=\tau(\mathcal{D})$, and number of upper regular cells $\mathcal{C}=\mathcal{C}(\mathcal{D})$. We remove all squares running along $\alpha$ with labels in $\mathcal{S}\subset\{0,1,2,\dotsc,\lfloor \frac{w}{2}\rfloor\}$ so that $|\mathcal{S}|=|w-2h|$ if $\alpha$ passes white squares, and  $|\mathcal{S}|=|w-2h+1|$, otherwise. Assume that the complement of $\mathcal{S}$ is $\{i_1,i_2,\dotsc,i_k\}$, for $1\leq i_1<i_2<\cdots<i_k\leq \lfloor \frac{w(D)}{2}\rfloor$.  We define $\mathcal{O}=\mathcal{O}_{\mathcal{D}}:=\{i_j: \text{$j$ is odd}\}$ and $\mathcal{E}=\mathcal{E}_{\mathcal{D}}:=\{i_j: \text{$j$ is even}\}$.

\medskip

(a) Assume that $\alpha$ passes white squares  and $w\geq 2h$. Then
\begin{equation}\label{symeq1}
\M^*(\mathcal{D}_{a}(\textbf{d};\mathcal{S}))=2^{\mathcal{C}-(w-1)h-\tau}\E(\mathcal{O})\Od(\mathcal{E})
\end{equation}
if $w$ is even;
\begin{equation}\label{symeq2}
\M^*(\mathcal{D}_{a}(\textbf{d};\mathcal{S}))=2^{\mathcal{C}-(w-1)h-\tau-1}\overline{\E}(\mathcal{O})\overline{\Od}(\mathcal{E})
\end{equation}
if $h$ and $w$ are odd;
\begin{equation}\label{symeq3}
\M^*(\mathcal{D}_{a}(\textbf{d};\mathcal{S}))=
2^{\mathcal{C}-(w-1)h-\tau}\overline{\E}(\mathcal{O})  \overline{\Od}(\mathcal{E})
\end{equation}
if $h$ is even and  $w$ is odd.

\medskip

(b) Assume that $\alpha$ passes black squares and $2h-1\geq w\geq h$. Then
\begin{equation}\label{symeq4}
\M^*(\mathcal{D}_{a}(\textbf{d};\mathcal{S}))=\frac{2^{\mathcal{C}-wh-\tau+\frac{w-1}{2}}}{(h-2)!}\overline{\E}(\mathcal{S}\cup \mathcal{E})\overline{\Od}(\mathcal{S}\cup \mathcal{O})
\end{equation}
if $h$ is even and $w$ is odd;
\begin{equation}\label{symeq5}
\M^*(\mathcal{D}_{a}(\textbf{d};\mathcal{S}))=\frac{2^{\mathcal{C}-wh-\tau+\frac{w-1}{2}}}{(h-1)!}\overline{\E}(\mathcal{S}\cup \mathcal{E})\overline{\Od}(\mathcal{S}\cup \mathcal{O})
\end{equation}
if $h$ and $w$ are odd;
\begin{equation}\label{symeq6}
\M^*(\mathcal{D}_{a}(\textbf{d};\mathcal{S}))=2^{\mathcal{C}-(w+1)h-\tau+\frac{w}{2}}\E(\mathcal{S}\cup \mathcal{E}) \Od(\mathcal{S}\cup \mathcal{O})
\end{equation}
if $w$ is even.
\end{thm}

We notice that if $\alpha$ passes white squares and $w<2h$, then the region $\mathcal{D}_{a}(\textbf{d};\mathcal{S})$ has no tiling (since the numbers of black cells and and white cells are not equal). Similarly, if  $\alpha$ passes black squares and $w>h$, then the region $\mathcal{D}_{a}(\textbf{d};\mathcal{S})$ has no tiling by the same reason. The condition $w\geq h$ ensures that the number of removed cells, i.e. $|\mathcal{S}|$, must be less than or equal the total number of cells on $\alpha$.

We consider next the centrally symmetric tilings of a symmetric quasi-hexagon \[\mathcal{H}_{a}(\textbf{d};\textbf{d}):=\mathcal{H}_{a}( d_1,d_2,\dotsc,d_k; d_1,d_2,\dotsc,d_k).\]
Define the  function
\begin{equation}\label{Maceq}
P(a,b,c):=\prod_{i=1}^{a}\prod_{j=1}^{b}\prod_{k=1}^c\frac{i+j+k-1}{i+j+k-2}.
\end{equation}
This is exactly the number of plane partitions fitting in an $a\times b\times c$  box \cite{McM}.

A \emph{regular} cell of a quasi-hexagon $\mathcal{H}$ is either a square or a triangle pointing away from the axis $\ell$. We notice that regular cells in a quasi-hexagon may be black or white (as opposed to being only black in the case of Douglas regions). Denote by $h_1(\mathcal{H})$ and $h_2(\mathcal{H})$ the number of rows of black regular cells above $\ell$ and the number rows of white regular cells below $\ell$, respectively. We call $h_1(\mathcal{H})$ and $h_2(\mathcal{H})$ the \emph{upper} and \emph{lower heights} of $\mathcal{H}$.  Denote by $\mathcal{C}_1(\mathcal{H})$ and $\mathcal{C}_2(\mathcal{H})$ the number of black regular cells above $\ell$ and the number of white regular cells below $\ell$, respectively. In the case when the quasi-hexagon $\mathcal{H}$ admits a southwest-to-northeast symmetry axis, we have $h_1(\mathcal{H})=h_2(\mathcal{H})$ and $\mathcal{C}_1(\mathcal{H})=\mathcal{C}_2(\mathcal{H})$.  The \emph{width} $w(\mathcal{H})$ of $\mathcal{H}$ is the number of cells running along each side of $\ell$.  We still call a row of an odd number of black triangles pointing toward $\ell$ and above $\ell$ a \textit{singular row} of $\mathcal{H}$. The number of singular rows $\tau(\mathcal{H})$ is also called the \textit{defect} of $\mathcal{H}$.

The number of centrally symmetric tilings of  a quasi-hexagon is given by the following theorem.

\begin{thm}\label{hexsym} Let $a$ be a positive integer and $\textbf{d}=(d_1,d_2,\dotsc, d_k)$ be a sequence of  positive integers, such that the symmetric quasi-hexagon $\mathcal{H}:=\mathcal{H}_{a}(\textbf{d};\textbf{d})$ has the heights $h=h_1(\mathcal{H})=h_2(\mathcal{H})$ less than or equal to the width $w$. Assume that $\mathcal{C}=\mathcal{C}_1(\mathcal{H})$ is the number of black regular cells above $\ell$ (and also the number of white regular cells below $\ell$ by the symmetry), and that $\tau$ is the defect of $\mathcal{H}$.

(a) If $h$ and $w$ are even, then
\begin{equation}
\M^*(\mathcal{H}_{a}(\textbf{d};\textbf{d}))=2^{\mathcal{C}-\frac{h(2w-h+1)}{2}-\tau} P\left(\frac{h}{2},\frac{h}{2},\frac{w-h}{2}\right)^2.
\end{equation}

(b) If $h$ is even and $w$ is odd, then
\begin{equation}
\M^*(\mathcal{H}_{a}(\textbf{d};\textbf{d}))=2^{\mathcal{C}-\frac{h(2w-h+1)}{2}-\tau} P\left(\frac{h}{2},\frac{h}{2},\frac{w-h-1}{2}\right)P\left(\frac{h}{2},\frac{h}{2},\frac{w-h+1}{2}\right).
\end{equation}

(c) If  $h$ is odd and $w$ is even, then
\begin{equation}
\M^*(\mathcal{H}_{a}(\textbf{d};\textbf{d}))=2^{\mathcal{C}-\frac{h(2w-h+1)}{2}-\tau} P\left(\frac{h-1}{2},\frac{h+1}{2},\frac{w-h}{2}\right)^2.
\end{equation}
\end{thm}

Note that if the width of the quasi-hexagon $\mathcal{H}_{a}(\textbf{d};\textbf{d})$ is less than the heights $h_1=h_2$, then it has no tiling (see Theorem 2.1 in \cite{Tri1}; it is easy to see that the width $w$ is equal to the value $a+m-n$ mentioned in this theorem).

It is worth noticing that Theorem \ref{hexsym} above generalizes Ciucu's Theorem 7.1 in \cite{Ciucu6}. The latter theorem in turn is a special case of Stanley's well-known enumeration of self-complementary plane partitions \cite[Eq. (3a)--(3c)]{Stanley2}.

\section{Preliminaries}\label{Background}

This section shares several preliminary results and definitions with the prequels   \cite{Tri1, Tri2} of the paper.  The first result not reported in \cite{Tri1,Tri2} is Ciucu's Lemma \ref{4cycle}.

A \textit{perfect matching} of a graph $G$ is a collection of edges such that each vertex of $G$ is adjacent to exactly one edge in the collection. The tilings of a region $\mathcal{R}$ can be naturally identified with the perfect matchings of its \textit{dual graph} (i.e., the graph whose vertices are the cells of $\mathcal{R}$, and whose edges connect two cells precisely when they share an edge). See Figures \ref{sampletiling} and \ref{sampletiling2} for the correspondence between tilings and perfect matchings. In the view of this, we denote the number of perfect matchings of a graph $G$ by $\M(G)$. More generally, if the edges of $G$ carry weights, $\M(G)$ denotes the sum of the weights of all perfect matchings of $G$, where the \emph{weight} of a perfect matching is the product of the weights of its constituent edges.


A \textit{forced edge} of a graph $G$ is an edge that is contained in every perfect matching of $G$. Let $G$ be a weighted graph with weight function $\wt$ on its edges, and $G'$ is obtained from $G$ by removing forced edges $e_1,\dotsc,e_k$, as well as the vertices incident to these edges\footnote{For the sake of simplicity, from now on, whenever we remove some forced edges, we remove also the vertices incident to them.}. Then one clearly has
\begin{equation}
\M(G)=\M(G')\prod_{i=1}^k\wt(e_i).
\end{equation}

We present next three basic preliminary results stated below.

\begin{figure}\centering
\begin{picture}(0,0)%
\includegraphics{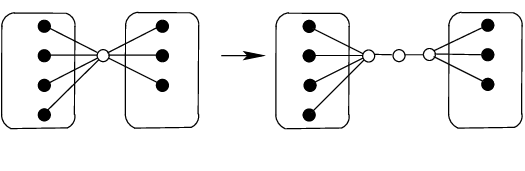}%
\end{picture}%
%
%
\setlength{\unitlength}{3947sp}%
\begingroup\makeatletter\ifx\SetFigFont\undefined%
\gdef\SetFigFont#1#2#3#4#5{%
  \reset@font\fontsize{#1}{#2pt}%
  \fontfamily{#3}\fontseries{#4}\fontshape{#5}%
  \selectfont}%
\fi\endgroup%
\begin{picture}(4188,1361)(593,-556)
\put(1336,591){\makebox(0,0)[lb]{\smash{{\SetFigFont{10}{12.0}{\familydefault}{\mddefault}{\updefault}{$v$}%
}}}}
\put(3549,621){\makebox(0,0)[lb]{\smash{{\SetFigFont{10}{12.0}{\familydefault}{\mddefault}{\updefault}{$v'$}%
}}}}
\put(3757, 89){\makebox(0,0)[lb]{\smash{{\SetFigFont{10}{12.0}{\familydefault}{\mddefault}{\updefault}{$x$}%
}}}}
\put(3999,621){\makebox(0,0)[lb]{\smash{{\SetFigFont{10}{12.0}{\familydefault}{\mddefault}{\updefault}{$v''$}%
}}}}
\put(820,-541){\makebox(0,0)[lb]{\smash{{\SetFigFont{10}{12.0}{\familydefault}{\mddefault}{\updefault}{$H$}%
}}}}
\put(1840,-535){\makebox(0,0)[lb]{\smash{{\SetFigFont{10}{12.0}{\familydefault}{\mddefault}{\updefault}{$K$}%
}}}}
\put(3031,-535){\makebox(0,0)[lb]{\smash{{\SetFigFont{10}{12.0}{\familydefault}{\mddefault}{\updefault}{$H$}%
}}}}
\put(4426,-484){\makebox(0,0)[lb]{\smash{{\SetFigFont{10}{12.0}{\familydefault}{\mddefault}{\updefault}{$K$}%
}}}}
\end{picture}%
\caption{Vertex splitting.}
\label{vertexsplitting}
\end{figure}
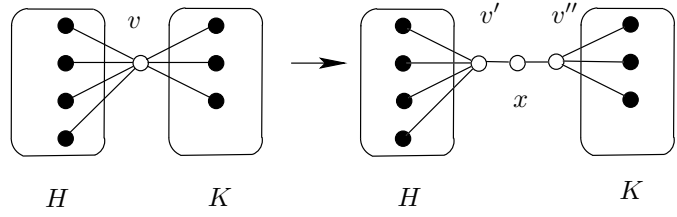

\begin{lem} [Vertex-Splitting Lemma;  Lemma 2.2 in \cite{Ciucu5} ]\label{VS}
 Let $G$ be a graph, $v$ be a vertex of it, and denote the set of neighbors of $v$ by $N(v)$.
  For an arbitrary disjoint union $N(v)=H\cup K$, let $G'$ be the graph obtained from $G\setminus v$ by including three new vertices $v'$, $v''$ and $x$ so that $N(v')=H\cup \{x\}$, $N(v'')=K\cup\{x\}$, and $N(x)=\{v',v''\}$ (see Figure \ref{vertexsplitting}). Then $\M(G)=\M(G')$.
\end{lem}

\begin{lem}[Star Lemma; Lemma 3.2 in  \cite{Tri1} ]\label{star}
Let $G$ be a weighted graph, and let $v$ be a vertex of~$G$. Let $G'$ be the graph obtained from $G$ by multiplying the weights of all edges that are adjacent to $v$ by a positive constant $t$. Then $\M(G')=t\M(G)$.
\end{lem}

Part (a) of the following result is a generalization due to Propp of the ``urban renewal" trick first observed by Kuperberg. Parts (b) and (c) are due to Ciucu (see Lemma 2.6 in [5]).

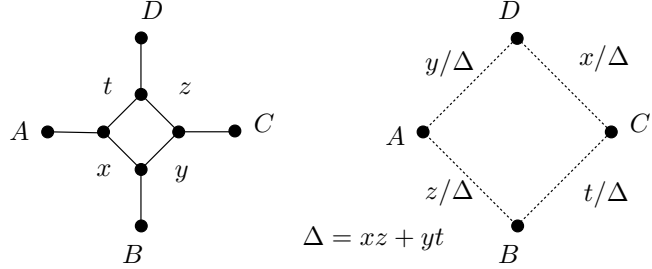
\begin{figure}\centering
\begin{picture}(0,0)%
\includegraphics{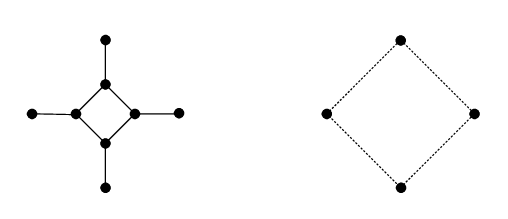}%
\end{picture}%
\setlength{\unitlength}{3947sp}%
\begingroup\makeatletter\ifx\SetFigFont\undefined%
\gdef\SetFigFont#1#2#3#4#5{%
  \reset@font\fontsize{#1}{#2pt}%
  \fontfamily{#3}\fontseries{#4}\fontshape{#5}%
  \selectfont}%
\fi\endgroup%
\begin{picture}(4054,1735)(340,-948)
\put(355,-156){\makebox(0,0)[lb]{\smash{{\SetFigFont{10}{12.0}{\familydefault}{\mddefault}{\updefault}{$A$}%
}}}}
\put(1064,-933){\makebox(0,0)[lb]{\smash{{\SetFigFont{10}{12.0}{\familydefault}{\mddefault}{\updefault}{$B$}%
}}}}
\put(1891,-106){\makebox(0,0)[lb]{\smash{{\SetFigFont{10}{12.0}{\familydefault}{\mddefault}{\updefault}{$C$}%
}}}}
\put(1182,603){\makebox(0,0)[lb]{\smash{{\SetFigFont{10}{12.0}{\familydefault}{\mddefault}{\updefault}{$D$}%
}}}}
\put(2717,-189){\makebox(0,0)[lb]{\smash{{\SetFigFont{10}{12.0}{\familydefault}{\mddefault}{\updefault}{$A$}%
}}}}
\put(3426,-933){\makebox(0,0)[lb]{\smash{{\SetFigFont{10}{12.0}{\familydefault}{\mddefault}{\updefault}{$B$}%
}}}}
\put(4253,-106){\makebox(0,0)[lb]{\smash{{\SetFigFont{10}{12.0}{\familydefault}{\mddefault}{\updefault}{$C$}%
}}}}
\put(3426,603){\makebox(0,0)[lb]{\smash{{\SetFigFont{10}{12.0}{\familydefault}{\mddefault}{\updefault}{$D$}%
}}}}
\put(904,-382){\makebox(0,0)[lb]{\smash{{\SetFigFont{10}{12.0}{\familydefault}{\mddefault}{\updefault}{$x$}%
}}}}
\put(1396,-388){\makebox(0,0)[lb]{\smash{{\SetFigFont{10}{12.0}{\familydefault}{\mddefault}{\updefault}{$y$}%
}}}}
\put(1418,130){\makebox(0,0)[lb]{\smash{{\SetFigFont{10}{12.0}{\familydefault}{\mddefault}{\updefault}{$z$}%
}}}}
\put(946,130){\makebox(0,0)[lb]{\smash{{\SetFigFont{10}{12.0}{\familydefault}{\mddefault}{\updefault}{$t$}%
}}}}
\put(2968,284){\makebox(0,0)[lb]{\smash{{\SetFigFont{10}{12.0}{\familydefault}{\mddefault}{\updefault}{$y/\Delta$}%
}}}}
\put(3934,311){\makebox(0,0)[lb]{\smash{{\SetFigFont{10}{12.0}{\familydefault}{\mddefault}{\updefault}{$x/\Delta$}%
}}}}
\put(3964,-544){\makebox(0,0)[lb]{\smash{{\SetFigFont{10}{12.0}{\familydefault}{\mddefault}{\updefault}{$t/\Delta$}%
}}}}
\put(2965,-526){\makebox(0,0)[lb]{\smash{{\SetFigFont{10}{12.0}{\familydefault}{\mddefault}{\updefault}{$z/\Delta$}%
}}}}
\put(2197,-817){\makebox(0,0)[lb]{\smash{{\SetFigFont{10}{12.0}{\familydefault}{\mddefault}{\updefault}{$\Delta= xz+yt$}%
}}}}
\end{picture}%
\caption{Urban renewal.}
\label{spider1}
\end{figure}
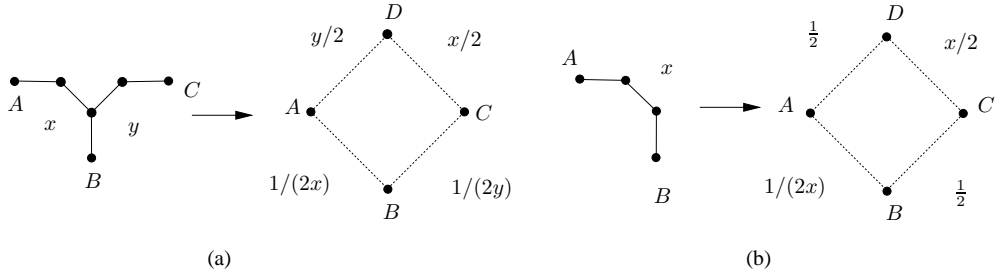
\begin{figure}\centering
\resizebox{!}{3.7cm}{
\begin{picture}(0,0)%
\includegraphics{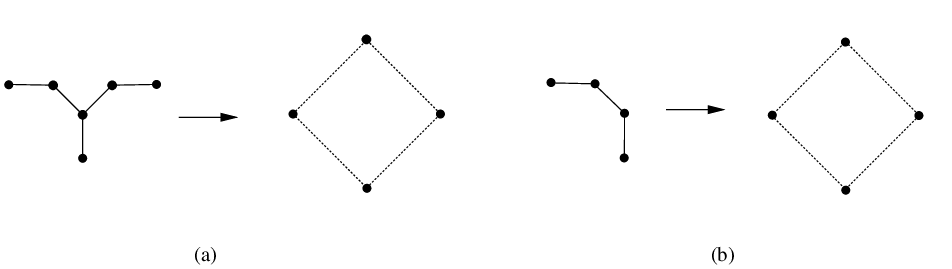}%
\end{picture}%
\setlength{\unitlength}{3947sp}%
\begingroup\makeatletter\ifx\SetFigFont\undefined%
\gdef\SetFigFont#1#2#3#4#5{%
  \reset@font\fontsize{#1}{#2pt}%
  \fontfamily{#3}\fontseries{#4}\fontshape{#5}%
  \selectfont}%
\fi\endgroup%
\begin{picture}(7605,2136)(222,-1352)
\put(237,-106){\makebox(0,0)[lb]{\smash{{\SetFigFont{10}{12.0}{\familydefault}{\mddefault}{\updefault}{$A$}%
}}}}
\put(828,-697){\makebox(0,0)[lb]{\smash{{\SetFigFont{10}{12.0}{\familydefault}{\mddefault}{\updefault}{$B$}%
}}}}
\put(1592,  0){\makebox(0,0)[lb]{\smash{{\SetFigFont{10}{12.0}{\familydefault}{\mddefault}{\updefault}{$C$}%
}}}}
\put(519,-249){\makebox(0,0)[lb]{\smash{{\SetFigFont{10}{12.0}{\familydefault}{\mddefault}{\updefault}{$x$}%
}}}}
\put(1164,-264){\makebox(0,0)[lb]{\smash{{\SetFigFont{10}{12.0}{\familydefault}{\mddefault}{\updefault}{$y$}%
}}}}
\put(2566,426){\makebox(0,0)[lb]{\smash{{\SetFigFont{10}{12.0}{\familydefault}{\mddefault}{\updefault}{$y/2$}%
}}}}
\put(3616,396){\makebox(0,0)[lb]{\smash{{\SetFigFont{10}{12.0}{\familydefault}{\mddefault}{\updefault}{$x/2$}%
}}}}
\put(2244,-714){\makebox(0,0)[lb]{\smash{{\SetFigFont{10}{12.0}{\rmdefault}{\mddefault}{\updefault}{$1/(2x)$}%
}}}}
\put(3646,-744){\makebox(0,0)[lb]{\smash{{\SetFigFont{10}{12.0}{\rmdefault}{\mddefault}{\updefault}{$1/(2y)$}%
}}}}
\put(2363,-129){\makebox(0,0)[lb]{\smash{{\SetFigFont{10}{12.0}{\familydefault}{\mddefault}{\updefault}{$A$}%
}}}}
\put(3131,600){\makebox(0,0)[lb]{\smash{{\SetFigFont{10}{12.0}{\familydefault}{\mddefault}{\updefault}{$D$}%
}}}}
\put(3830,-174){\makebox(0,0)[lb]{\smash{{\SetFigFont{10}{12.0}{\rmdefault}{\mddefault}{\updefault}{$C$}%
}}}}
\put(3119,-933){\makebox(0,0)[lb]{\smash{{\SetFigFont{10}{12.0}{\rmdefault}{\mddefault}{\updefault}{$B$}%
}}}}
\put(4489,249){\makebox(0,0)[lb]{\smash{{\SetFigFont{10}{12.0}{\rmdefault}{\mddefault}{\updefault}{$A$}%
}}}}
\put(5198,-814){\makebox(0,0)[lb]{\smash{{\SetFigFont{10}{12.0}{\rmdefault}{\mddefault}{\updefault}{$B$}%
}}}}
\put(6150,-115){\makebox(0,0)[lb]{\smash{{\SetFigFont{10}{12.0}{\rmdefault}{\mddefault}{\updefault}{$A$}%
}}}}
\put(6977,-942){\makebox(0,0)[lb]{\smash{{\SetFigFont{10}{12.0}{\rmdefault}{\mddefault}{\updefault}{$B$}%
}}}}
\put(7686,-115){\makebox(0,0)[lb]{\smash{{\SetFigFont{10}{12.0}{\rmdefault}{\mddefault}{\updefault}{$C$}%
}}}}
\put(6977,594){\makebox(0,0)[lb]{\smash{{\SetFigFont{10}{12.0}{\rmdefault}{\mddefault}{\updefault}{$D$}%
}}}}
\put(5251,179){\makebox(0,0)[lb]{\smash{{\SetFigFont{10}{12.0}{\rmdefault}{\mddefault}{\updefault}{$x$}%
}}}}
\put(6369,441){\makebox(0,0)[lb]{\smash{{\SetFigFont{10}{12.0}{\rmdefault}{\mddefault}{\updefault}{$\frac12$}%
}}}}
\put(7501,-781){\makebox(0,0)[lb]{\smash{{\SetFigFont{10}{12.0}{\rmdefault}{\mddefault}{\updefault}{$\frac12$}%
}}}}
\put(6046,-736){\makebox(0,0)[lb]{\smash{{\SetFigFont{10}{12.0}{\rmdefault}{\mddefault}{\updefault}{$1/(2x)$}%
}}}}
\put(7426,374){\makebox(0,0)[lb]{\smash{{\SetFigFont{10}{12.0}{\rmdefault}{\mddefault}{\updefault}{$x/2$}%
}}}}
\end{picture}}
\caption{Two variants of the urban renewal.}
\label{spider2}
\end{figure}

\begin{lem} [Spider Lemma]\label{spider}
(a) Let $G$ be a weighted graph containing the subgraph $K$ shown on the left in Figure \ref{spider1} (the labels indicate weights, unlabeled edges have weight 1). Suppose in addition that the four inner black vertices in the subgraph $K$, different from $A,B,C,D$, have no neighbors outside $K$. Let $G'$ be the graph obtained from $G$ by replacing $K$ by the graph $\overline{K}$ shown on right in Figure \ref{spider}, where the dashed lines indicate new edges, weighted as shown. Then $\M(G)=(xz+yt)\M(G')$.

(b) Consider the above local replacement operation when $K$ and $\overline{K}$ are graphs shown in Figure \ref{spider2}(a) with the indicated weights (in particular, $K'$ has a new vertex $D$, that is incident only to $A$ and $C$). Then $\M(G)=2\M(G')$.

(c) The statement of part (b) is also true when $K$ and $\overline{K}$ are the graphs indicated in Figure \ref{spider2}(b). (In this case $G'$ has two new vertices $C$ and $D$, that are adjacent only to one another and to $B$ and $A$, respectively).
\end{lem}

We quote the following useful result of Ciucu \cite{Ciucu1}.
\begin{figure}\centering
\begin{picture}(0,0)%
\includegraphics{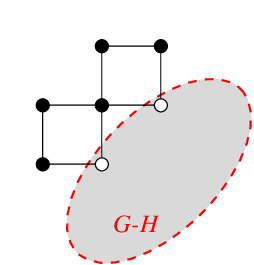}%
\end{picture}%
%
%
\setlength{\unitlength}{3947sp}%
\begingroup\makeatletter\ifx\SetFigFont\undefined%
\gdef\SetFigFont#1#2#3#4#5{%
  \reset@font\fontsize{#1}{#2pt}%
  \fontfamily{#3}\fontseries{#4}\fontshape{#5}%
  \selectfont}%
\fi\endgroup%
\begin{picture}(2025,2102)(841,-1384)
\put(1489,-40){\makebox(0,0)[lb]{\smash{{\SetFigFont{12}{14.4}{\rmdefault}{\mddefault}{\itdefault}{$a$}%
}}}}
\put(1606,482){\makebox(0,0)[lb]{\smash{{\SetFigFont{12}{14.4}{\rmdefault}{\mddefault}{\itdefault}{$b_1$}%
}}}}
\put(2104,494){\makebox(0,0)[lb]{\smash{{\SetFigFont{12}{14.4}{\rmdefault}{\mddefault}{\itdefault}{$b_2$}%
}}}}
\put(2155,-277){\makebox(0,0)[lb]{\smash{{\SetFigFont{12}{14.4}{\rmdefault}{\mddefault}{\itdefault}{$b_3$}%
}}}}
\put(1606,-775){\makebox(0,0)[lb]{\smash{{\SetFigFont{12}{14.4}{\rmdefault}{\mddefault}{\itdefault}{$c_3$}%
}}}}
\put(856,-709){\makebox(0,0)[lb]{\smash{{\SetFigFont{12}{14.4}{\rmdefault}{\mddefault}{\itdefault}{$c_2$}%
}}}}
\put(859,-55){\makebox(0,0)[lb]{\smash{{\SetFigFont{12}{14.4}{\rmdefault}{\mddefault}{\itdefault}{$c_1$}%
}}}}
\end{picture}%
\caption{Illustrating Lemma \ref{4cycle}.}
\label{4cyclelemma}
\end{figure}

\begin{lem}[Lemma 4.2 in \cite{Ciucu1} ]\label{4cycle}
Let $G$ be a weighted graph having a $7$-vertex induced subgraph $H$ consisting of two $4$-cycles that share a vertex. Let $a$, $b_1$, $b_2$, $b_3$ and $a$, $c_1$, $c_2$, $c_3$ be the vertices of the 4-cycles (listed in cyclic order) and suppose $b_3$ and $c_3$ are only the vertices of $H$ with the neighbors outside $H$ (see Figure \ref{4cycle}). Assume that the product of weights of opposite edges in each $4$-cycle of $H$ is a constant $d$, that is
\begin{align}
wt(b_1,b_2)wt(a_,b_3)=wt(b_2,b_3)wt(a,b_1)=wt(c_1,c_2)wt(a,c_3)=wt(c_2,c_3)wt(a,c_1)=d.
\end{align}
Here we use the notation $wt(a,b)$ for the weight of the edge connecting the two vertices $a$ and $b$.
Let $G'$ be the subgraph of $G$ obtained by deleting $b_1$, $b_2$, $c_1$ and $c_2$, weighted by restriction. Then
\[\M(G)=2wt(b_1,b_2)wt(c_1,c_2) \M(G').\]
\end{lem}

\begin{figure}\centering
\begin{picture}(0,0)%
\includegraphics{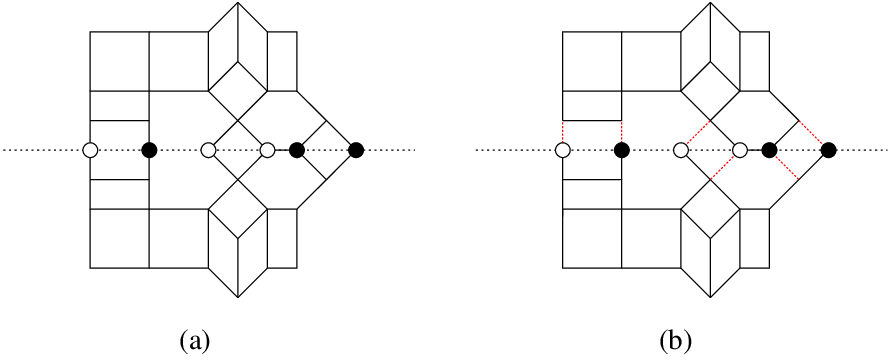}%
\end{picture}%
%
%
\setlength{\unitlength}{3947sp}%
\begingroup\makeatletter\ifx\SetFigFont\undefined%
\gdef\SetFigFont#1#2#3#4#5{%
  \reset@font\fontsize{#1}{#2pt}%
  \fontfamily{#3}\fontseries{#4}\fontshape{#5}%
  \selectfont}%
\fi\endgroup%
\begin{picture}(7111,2872)(225,-2257)
\put(4135,-401){\makebox(0,0)[lb]{\smash{{\SetFigFont{14}{16.8}{\rmdefault}{\mddefault}{\updefault}{$\ell$}%
}}}}
\put(6379,-756){\makebox(0,0)[lb]{\smash{{\SetFigFont{10}{12.0}{\rmdefault}{\mddefault}{\updefault}{$a_3$}%
}}}}
\put(6911,-519){\makebox(0,0)[lb]{\smash{{\SetFigFont{10}{12.0}{\rmdefault}{\mddefault}{\updefault}{$b_3$}%
}}}}
\put(4253,189){\makebox(0,0)[lb]{\smash{{\SetFigFont{14}{16.8}{\rmdefault}{\mddefault}{\updefault}{$G^+$}%
}}}}
\put(4312,-1287){\makebox(0,0)[lb]{\smash{{\SetFigFont{14}{16.8}{\rmdefault}{\mddefault}{\updefault}{$G^-$}%
}}}}
\put(6143,-492){\makebox(0,0)[lb]{\smash{{\SetFigFont{10}{12.0}{\rmdefault}{\mddefault}{\updefault}{$\frac{1}{2}$}%
}}}}
\put(532,-933){\makebox(0,0)[lb]{\smash{{\SetFigFont{14}{16.8}{\rmdefault}{\mddefault}{\updefault}{$G$}%
}}}}
\put(296,-460){\makebox(0,0)[lb]{\smash{{\SetFigFont{14}{16.8}{\rmdefault}{\mddefault}{\updefault}{$\ell$}%
}}}}
\put(710,-460){\makebox(0,0)[lb]{\smash{{\SetFigFont{10}{12.0}{\rmdefault}{\mddefault}{\updefault}{$a_1$}%
}}}}
\put(1477,-519){\makebox(0,0)[lb]{\smash{{\SetFigFont{10}{12.0}{\rmdefault}{\mddefault}{\updefault}{$b_1$}%
}}}}
\put(1714,-755){\makebox(0,0)[lb]{\smash{{\SetFigFont{10}{12.0}{\rmdefault}{\mddefault}{\updefault}{$a_2$}%
}}}}
\put(2304,-492){\makebox(0,0)[lb]{\smash{{\SetFigFont{10}{12.0}{\rmdefault}{\mddefault}{\updefault}{$b_2$}%
}}}}
\put(2540,-814){\makebox(0,0)[lb]{\smash{{\SetFigFont{10}{12.0}{\rmdefault}{\mddefault}{\updefault}{$a_3$}%
}}}}
\put(3131,-519){\makebox(0,0)[lb]{\smash{{\SetFigFont{10}{12.0}{\rmdefault}{\mddefault}{\updefault}{$b_3$}%
}}}}
\put(4490,-460){\makebox(0,0)[lb]{\smash{{\SetFigFont{10}{12.0}{\rmdefault}{\mddefault}{\updefault}{$a_1$}%
}}}}
\put(5257,-519){\makebox(0,0)[lb]{\smash{{\SetFigFont{10}{12.0}{\rmdefault}{\mddefault}{\updefault}{$b_1$}%
}}}}
\put(5494,-755){\makebox(0,0)[lb]{\smash{{\SetFigFont{10}{12.0}{\rmdefault}{\mddefault}{\updefault}{$a_2$}%
}}}}
\put(5848,-610){\makebox(0,0)[lb]{\smash{{\SetFigFont{10}{12.0}{\rmdefault}{\mddefault}{\updefault}{$b_2$}%
}}}}
\end{picture}
\caption{(a) A graph $G$ with a horizontal symmetric axis; (b) the resulting graph after the cutting procedure.}
\label{verticalfactor}
\end{figure}
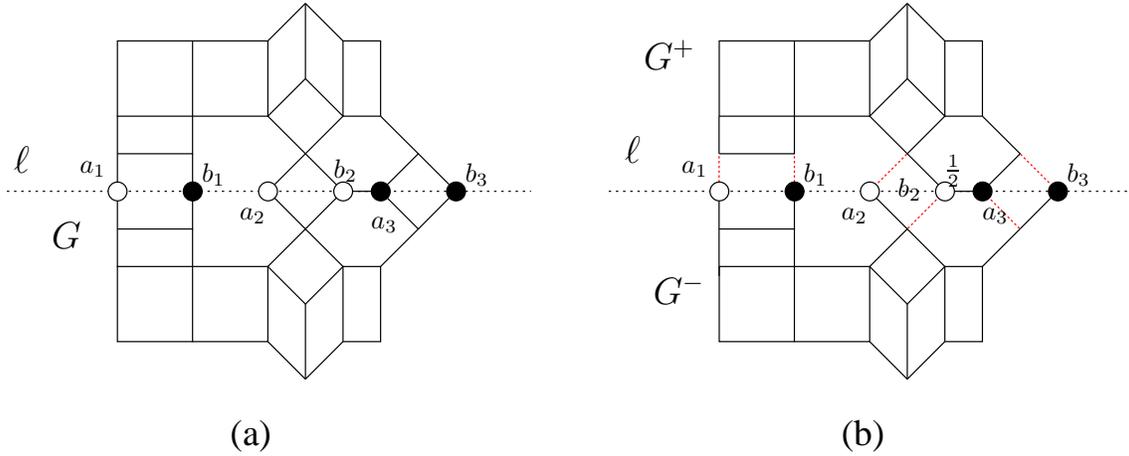

Next, we present a powerful tool in enumeration of perfect matchings of reflectively symmetric graphs. This was first introduced by Ciucu \cite{Ciucu}.

 Let $G$ be a weighted planar bipartite graph that is symmetric about a horizontal line $\ell$. Assume that the set of vertices lying on $\ell$ is a cut set of $G$ (i.e., the removal of these vertices disconnects $G$). One readily sees that the number of vertices of $G$ on $\ell$ must be even if $G$ has perfect matchings, let $\eta(G)$ be half of this number. Let $a_1,b_1,a_2,b_2,\dots,a_{\eta(G)},b_{\eta(G)}$ be the vertices lying on $\ell$, as they occur from left to right. Color vertices of $G$ by black or white so that any two adjacent vertices have opposite colors. Without loss of generality, we assume that $a_1$ is always colored white. Delete all edges above $\ell$ at all white $a_i$'s and black $b_j$'s, and delete all edges below $\ell$ at all black $a_i$'s and white $b_j$'s. Reduce the weight of each edge lying on $\ell$ by half; leave all other weights unchanged. Since the set of vertices of $G$ on $\ell$ is a cut set, the graph obtained from the above cutting procedure has two disconnected parts, one above $\ell$ and one below $\ell$, denoted by $G^+$ and $G^-$, respectively (see Figure \ref{verticalfactor}).

\begin{thm}[Ciucu's Factorization Theorem \cite{Ciucu}]\label{factorthm}
Let $G$ be a bipartite weighted symmetric graph separated by its symmetry axis. Then
\begin{equation}\label{factoreq}
\M(G)=2^{\eta(G)}\M(G^+)\M(G^-).
\end{equation}
\end{thm}

Consider a $(2m+1)\times(2n+1)$ rectangular chessboard and suppose the corners are black. The \textit{Aztec rectangle graph} $AR_{m,n}$ is the graph whose vertices are the white unit squares and whose edges connect precisely those pairs of white unit squares that are diagonally adjacent (see Figure \ref{fourAR}(a) for $AR_{3,5}$). The \textit{odd Aztec rectangle graph} $OR_{m,n}$ is the graph whose vertices are the black unit squares whose edges connect precisely those pairs of black unit squares that are diagonally adjacent (see Figure \ref{fourAR}(b) for $OR_{3,5}$).  If one removes all the bottommost vertices in $AR_{m,n}$, the resulting graph is denoted by $AR_{m-\frac12,n}$, and called a \textit{baseless Aztec rectangle} (see Figure \ref{fourAR}(c) for $AR_{5/2,5}$). We also consider the graph $AR_{m,n-\frac12}$ that is obtained from the Aztec rectangle $AR_{m,n}$ by removing all its leftmost vertices (see Figure \ref{fourAR}(d) for $AR_{3,9/2}$).

It is worth noticing that when $n=m$, the Aztec rectangle graph $AR_{m,n}$ becomes the \emph{Aztec diamond graph} $AD_{n}$. Elkies, Kuperberg, Larsen and Propp \cite{Elkies1,Elkies2} showed that the number of perfect matchings of $AD_{n}$ is exactly $2^{n(n+1)/2}$. The Aztec rectangle  graph $AR_{m,n}$  does \emph{not} have perfect matchings in general, however, when certain vertices have been removed from one of its sides, the perfect matchings are enumerated by a simple product formula (see e.g. Proposition 2.1 in \cite{Cohn}).

\begin{figure}\centering
\includegraphics[width=8cm]{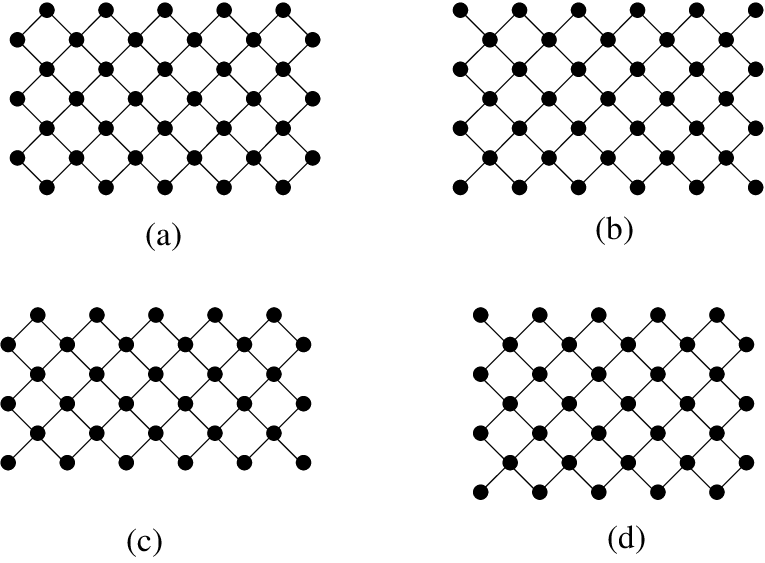}
\caption{Four types of Aztec rectangle graphs.}
\label{fourAR}
\end{figure}

\medskip

Next, we consider several variations of the Aztec rectangles\footnote{From now on we use the word ``\emph{Aztec rectangle(s)}" to mean ``\emph{Aztec rectangle graph(s)}".} as follows.

Label the vertices on the left side of the Aztec rectangle $AR_{m,n}$ from bottom up by $1,2,3,\dotsc,m$. Denote by $AR^{o}_{m,n}$ and $AR^{e}_{m,n}$ the graphs obtained from $AR_{m,n}$ by removing all odd-labeled and all even-labeled vertices, respectively (see Figures \ref{symT5}(b) and (d) for $AR^{o}_{6,5}$ and $AR^{o}_{5,5}$,  and Figures \ref{symT6}(b) and (d) for $AR^{e}_{5,5}$ and $AR^{e}_{6,5}$). We call $AR^{o}_{m,n}$ and $AR^{e}_{m,n}$ the \emph{odd-} and \emph{even-trimmed versions} of $AR_{m,n}$, respectively.

Applying a similar process, we obtain the odd- and even-trimmed versions of the graphs  $OR_{m,n}$, $AR_{m-\frac12,n}$, and $AR_{m,n-\frac12}$. Figures \ref{symT2}(b) and (d) illustrate the graph $OR^{o}_{6,5}$ and $OR^{o}_{5,5}$; while the graphs $OR^{e}_{5,5}$ and $OR^{e}_{6,5}$ are shown in  Figures \ref{symT1}(b) and (d). See Figures \ref{symT3}(b) and (d) for $AR^{o}_{5,9/2}$ and $AR^{o}_{6,9/2}$, and Figures \ref{symT4}(b) and (d) for $AR^{e}_{6,9/2}$ and $AR^{e}_{5,9/2}$. Finally,  examples of $AR^{o}_{m,n}$ and $AR^{e}_{m,n}$  are illustrated in  Figures \ref{symT8}(b) and (d) and in Figures \ref{symT7}(b) and (d), respectively.

Similar to the case of the Aztec rectangles,  the above trimmed Aztec rectangles  do not have perfect matchings in general, and we are interested in the case in which some bottom vertices of them have been removed.

Label the bottom vertices of $AR^{e}_{m,n}$,  $AR^{o}_{m,n}$, $OR^{e}_{m,n}$, and $OR^{o}_{m,n}$ by $1,2,\dotsc,n$ from left to right. For $0\leq k\leq n$ and $1\leq a_1<a_2<\cdots<a_k\leq n$, define $AR^{e}_{m,n}(a_1,a_2,\dotsc,a_k)$ (resp., $AR^{o}_{m,n}(a_1,a_2,\dotsc,a_k)$) to be the graph obtained from $AR^{e}_{m,n}$ (resp., $AR^{o}_{m,n}$) by removing all bottom vertices, except for the ones at the positions $a_1,a_2,\dotsc,a_k$. Define $OR^{e}_{m,n}(a_1,a_2,\dotsc,a_k)$ (resp., $OR^{o}_{m,n}(a_1,a_2,\dotsc,a_k)$) to be the graph obtained from $OR^{e}_{m,n}$ (resp., $OR^{o}_{m,n}$) by removing the bottom vertices at the positions $a_1,a_2,\dotsc,a_k$.

Similarly, we label the bottom vertices of $AR^{e}_{m,n+\frac12}$ and $AR^{e}_{m-\frac12,n}$ by $0,1,2,\dotsc,n$ from left to right; and we also label the bottom vertices of  $AR^{o}_{m,n+\frac12}$, and $AR^{o}_{m-\frac12,n}$ by $1,2,\dotsc,n$. For $0\leq k\leq n$ and $1\leq a_1<a_2<\cdots<a_k\leq n$, define $AR^{e}_{m,n+\frac12}(a_1,a_2,\dotsc,a_k)$ (resp., $AR^{o}_{m,n+\frac12}(a_1,a_2,\dotsc,a_k)$) to be the graph obtained from $AR^{e}_{m,n+\frac12}$ (resp., $AR^{o}_{m,n+\frac12}$) by removing all bottom vertices, except for the ones at the positions $a_1,a_2,\dotsc,a_k$.  The graph $AR^{e}_{m-\frac12,n}(a_1,a_2,\dotsc,a_k)$ (resp., $AR^{o}_{m-\frac12,n}(a_1,a_2,\dotsc,a_k)$) is the graph obtained from $AR^{e}_{m-\frac12,n}$ (resp., $AR^{o}_{m-\frac12,n}$) by removing the bottom vertices at the positions $0$ and $a_1,a_2,\dotsc,a_k$ (resp., at the positions $a_1,a_2,\dotsc,a_k$).

The author showed that perfect matchings of a trimmed Aztec rectangle are always enumerated by a simple product formula (see Theorems 1.2 and 1.3 in \cite{Tri3}; strictly speaking, our graphs here are the dual graphs of the regions in these theorems).

\begin{thm}\label{QARthm}
For any $1\leq k<n$ and $1\leq a_1<a_2<\cdots<a_k\leq n$
\begin{align}\label{main1}
\M(AR^{e}_{2k-1,n}(a_1,a_2,\dotsc,a_k))=\M(AR^{e}_{2k,n}(a_1,a_2,\dotsc,a_k))=\E(a_1,a_2,\dotsc,a_k),
\end{align}
\begin{align}\label{main2}
\M(AR^{o}_{2k,n}(a_1,a_2,\dotsc,a_k))=\M(AR^{o}_{2k+1,n}(a_1,a_2,\dotsc,a_k))=\Od(a_1,a_2,\dotsc,a_k),
\end{align}
\begin{align}\label{main3}
\M(OR^{e}_{2k,n}(a_1,a_2,\dotsc,a_{k}))=\M(OR^{e}_{2k+1,n}(a_1,a_2,\dotsc,a_{k}))=2^{-k}\Od(a_1,a_2,\dotsc,a_k),
\end{align}
\begin{align}\label{main4}
\M(OR^{o}_{2k-1,n}(a_1,a_2,\dotsc,a_{k}))=\M(OR^{o}_{2k,n}(a_1,a_2,\dotsc,a_{k}))=2^{-k}\E(a_1,a_2,\dotsc,a_k),
\end{align}
\begin{align}\label{VReq2}
\M(AR^{o}_{2k,n-\frac12}(a_1,a_2,\dotsc,a_k))=\M(AR^{o}_{2k+1,n-\frac12}(a_1,a_2,\dotsc,a_k))=2^{k}\overline{\Od}(a_1,a_2,\dotsc,a_k),
\end{align}
\begin{align}\label{VReq1}
\M(AR^{e}_{2k-1,n-\frac12}(a_1,a_2,\dotsc,a_k))=\M(AR^{e}_{2k,n-\frac12}(a_1,a_2,\dotsc,a_k))=2^{-k}\overline{\E}(a_1,a_2,\dotsc,a_k),
\end{align}
\begin{align}\label{VReq4}
\M(AR^{o}_{2k-\frac12,n}(a_1,a_2,\dotsc,a_{k}))=\M(AR^{o}_{2k+\frac12,n}(a_1,a_2,\dotsc,a_k))=\overline{\Od}(a_1,a_2,\dotsc,a_k),
\end{align}
and
\begin{align}\label{VReq3}
\M(AR^{e}_{2k+\frac12,n}(a_1,a_2,\dotsc,a_{k}))=\M(AR^{e}_{2k+3/2,n}(a_1,a_2,\dotsc,a_k))=\frac{1}{(2k)!}\overline{\E}(a_1,a_2,\dotsc,a_k).
\end{align}
\end{thm}

\section{Centrally symmetric matchings of an Aztec rectangle with holes}\label{Mainproof1}

In his Ph.D. thesis \cite{Yang}, Bo-Yin Yang  proved a conjecture posed by Jockush on the number of centrally symmetric tilings of the Aztec diamond region. Ciucu reproved the result in \cite{Ciucu} by using his own factorization thorem (Theorem \ref{factorthm}) and a tiling enumeration of  Jockush  and Propp \cite{JP}.   It is worth noticing that the author gave a new proof for  Jockush--Propp's enumeration in \cite{Tri5}, and also generalized it  in \cite{Tri3}. In this section, we enumerate centrally symmetric perfect matchings of an Aztec rectangle with several vertices removed along the symmetry axis (we also call these removed vertices \emph{holes}). Our result implies Ciucu and Yang's previous work as a special case when the set of removed vertices is empty (and the Aztec rectangle becomes an Aztec diamond graph).


Consider an Aztec rectangle $AR_{m,n}$ with the horizontal symmetry axis $\ell$ and the vertical symmetry axis $\ell'$. We label the vertices of $AR_{a,b}$ on $\ell$  as follows. If the symmetry center of the graph is a vertex on $\ell$, then we label it by $0$. Label two vertices that are closest to the center by $1$, label the second closest vertices by $2$, and  so on. We remove several vertices so that the resulting graph still admits the vertical symmetry axis $\ell'$. Denote by $\mathcal{S}$  the label set of removed vertices, which are not the center, and denote by $AR_{m,n}(\mathcal{S})$ the resulting graph. Assume that $\{i_1,i_2,\dotsc,i_k\}$ is the label set of the vertices of $AR_{m,n}(\mathcal{S})$ on $\ell$. It is easy to see that if a bipartite graph has perfect matchings, then it must have the same number of vertices in the two vertex classes. This implies that, in any cases,  $|\mathcal{S}|=|m-n|$. Moreover, for even $m$, the graph $AR_{m,n}(\mathcal{S})$ has perfect matchings only if  $m\leq n$; for odd $m$, the graph $AR_{m,n}(\mathcal{S})$ has perfect matchings only if $m\geq n$. In the latter case we also have $n\geq m/2$, since the number of removed vertices $m-n$ must be less than or equal to the number of vertices in $\ell$.  In particular, Ciucu showed that if $m$ is even and $m\leq n$, then the number of perfect matchings of  $AR_{m,n}(\mathcal{S})$ is given by a simple product formula (see Theorem 4.1 in \cite{Ciucu}).

 We are interested in the centrally symmetric perfect matchings of $AR_{m,n}(\mathcal{S})$, i.e. the perfect matchings which are invariant under the $180^\circ$ rotation around the symmetry center of the graph. Denote by $\M^*(G)$ the number of centrally perfect matchings of a graph $G$. We separate the label set $\{i_1,i_2,\dotsc,i_k\}$ of the vertices of $AR_{m,n}(\mathcal{S})$ on $\ell$ into two subsets:  $\mathcal{O}:=\{i_j: \text{$j$ is odd}\}$ and $\mathcal{E}:=\{i_j: \text{$j$ is even}\}$.

The number of centrally symmetric perfect matchings of an Aztec rectangle graph with `holes' $AR_{m,n}(\mathcal{S})$ is given by simple products in the following theorem.
\begin{thm}\label{sym}
(a) For any $n>m$ and $\mathcal{S}=\{s_1,s_2,\dotsc,s_{n-m}\}$
\begin{equation}\label{STeq1}
\M^{*}(AR_{2m,2n}(\mathcal{S}))=
2^{m}\E(\mathcal{O})\Od(\mathcal{E})
\end{equation}
(b) For any $m>n>m/2$ and $\mathcal{S}=\{ s_1,s_2,\dotsc,s_{m-n}\}$
\begin{equation}\label{STeq2}
\M^{*}(AR_{2m-1,2n-1}(\mathcal{S}))=
2^{n-m}\E(\mathcal{S}\cup \mathcal{E}) \Od(\mathcal{S}\cup \mathcal{O})\\
\end{equation}
(c) For any $m>n>m/2$ and $\mathcal{S}=\{ s_1,s_2,\dotsc,s_{m-n-1}\}$
\begin{equation}\label{STeq3}
\M^{*}(AR_{2m-1,2n}(\mathcal{S}))=
\begin{cases}
\frac{2^n}{(m-2)!}\overline{\E}(\mathcal{S}\cup \mathcal{E})\overline{\Od}(\mathcal{S}\cup \mathcal{O})& \text{if $m$ is even;}\\
\frac{2^n}{(m-1)!}\overline{\E}(\mathcal{S}\cup \mathcal{E})\overline{\Od}(\mathcal{S}\cup \mathcal{O})& \text{if $m$ is odd;}
\end{cases}
\end{equation}
(d) For any $n>m$ and $\mathcal{S}=\{s_1,s_2,\dotsc,s_{n-m-1}\}$
\begin{equation}\label{STeq4}
\M^{*}(AR_{2m,2n-1}(\mathcal{S}))=
\begin{cases}
2^{m-1}\overline{\E}(\mathcal{O})\overline{\Od}(\mathcal{E}) & \text{if $m$ is odd;}\\
2^{m}\overline{\E}(\mathcal{O})  \overline{\Od}(\mathcal{E})& \text{if $m$ is even.}
\end{cases}
\end{equation}
\end{thm}

\begin{figure}\centering
\includegraphics[width=13.5cm]{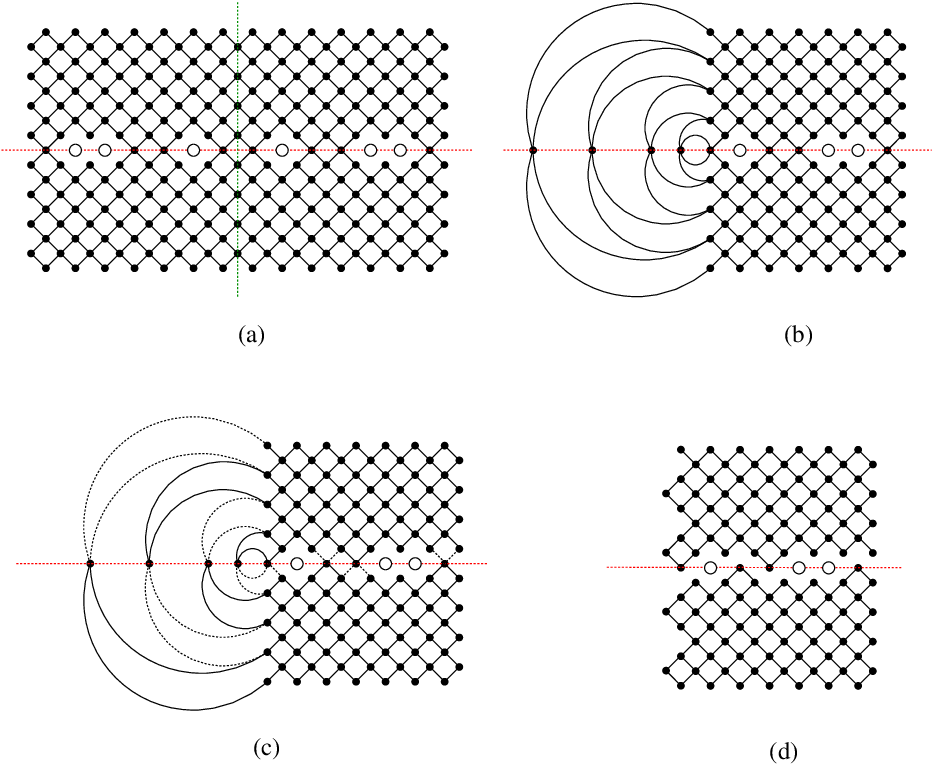}
\caption{Illustrating of the proof of Theorem \ref{sym}.}
\label{Aztechole}
\end{figure}



\begin{proof}[Proof of Theorem \ref{sym}]
We only prove in detail part (a), as the other parts can be obtained in a completely analogous manner. We will use Ciucu's  Factorization Theorem (Theorem \ref{factorthm}) to show that the number of centrally symmetric perfect matchings of our graph is given by a certain product of the numbers of perfect matchings of two graphs in Theorem \ref{QARthm}.

Consider the Aztec rectangle with holes $G=AR_{2m,2n}(\mathcal{S})$ with the horizontal and vertical symmetry axes $\ell$ and $\ell'$ (see Figure \ref{Aztechole}(a) for $AR_{8,14}(2,5,6)$). In this case, we have $n\geq m$ and  $|\mathcal{S}|=n-m$. Consider the subgraph $H$ of $G$ that is induced by vertices lying on $\ell'$ or staying on the right of $\ell'$. Label the vertices of $G$ on $\ell'$ which are staying above the horizontal axis $\ell$ by $v_1,v_2,v_3,\dotsc$ from bottom to top; and label the vertices of $G$ on $\ell'$ which are below $\ell$ by $u_1,u_2,u_3,\dotsc$ from top to bottom.

 It is easy to see that each centrally symmetric perfect matching $\mu$ of $G$ is determined uniquely by its sub-matching $\mu'$ restricted to the edge set of $H$, i.e., $\mu':=\mu \cap E(H)$. On the other hand, by the symmetry of $\mu$, exactly one of two vertices $u_i$ and $v_i$ is covered by $\mu'$. Therefore, the sub-matching $\mu'$ corresponds to a perfect matching of the graph $\widetilde{H}$  obtained from $H$ by identifying $u_i$ and $v_i$, for any $i=1,2,3,\dotsc$.  This implies that the centrally symmetric perfect matchings of $G$ are in bijection with the perfect matchings of $\widetilde{H}$.

Moreover, we can put the vertices in $\widetilde{H}$ which are obtained by identifying $u_i$ and $v_i$ on the horizontal axis $\ell$, so that $\widetilde{H}$ has $\ell$ as its horizontal symmetry axis (see Figure  \ref{Aztechole}(b)). By Ciucu's Factorization Theorem (Theorem \ref{factorthm}), we have
\begin{equation}
\M(\widetilde{H})=2^{m}\M(\widetilde{H}^+)\M(\widetilde{H}^-),
\end{equation}
where $\widetilde{H}$ has exactly $2m$ vertices on $\ell$ (see the cutting procedure in Figures \ref{Aztechole}(c) and (d)).

For even $m$, we have $\widetilde{H}^+$ is isomorphic to $AR^{e}_{m,n}(\mathcal{O})$ and $\widetilde{H}^-$ is isomorphic to $AR^{o}_{m,n}(\mathcal{E})$ (see Figure \ref{Aztechole}(d)), while $\widetilde{H}^+$ is isomorphic to $AR^{o}_{m,n}(\mathcal{E})$ and $\widetilde{H}^-$ is isomorphic to $AR^{e}_{m,n}(\mathcal{O})$ when $m$ is odd. Therefore, (\ref{STeq1}) follows from Theorem \ref{QARthm}. This finishes our proof.
%
%
%
\end{proof}

\section{Symmetric tilings of  Douglas regions}\label{Mainproof2}

In the first part of this section, we present several new subgraph replacement rules that will be employed in the proof of Theorem \ref{symthm1}.

The \emph{connected sum} $G\#G'$ of two disjoint graphs $G$ and $G'$ along the ordered sets of vertices $\{v_1,\dots,v_n\}\subset V(G)$ and $\{v'_1,\dots,v'_n\}\subset V(G')$ is the graph obtained from $G$ and $G'$ by identifying vertices $v_i$ and $v'_i$, for $i=1,2,\dots,n$.

In the next lemmas (Lemmas \ref{newT1}, \ref{newT2}, \ref{newT3}, and \ref{newT4}),  we always assume that $G$ is a graph, and $\{v_1,v_2,\dotsc,v_n\}$ is an ordered set of its vertices. Moreover, all connected sums
act on $G$ along $\{v_1,v_2,\dotsc,v_n\}$ and on other summands along their bottommost vertices ordered from left to right.

\begin{figure}\centering
\includegraphics[width=12cm]{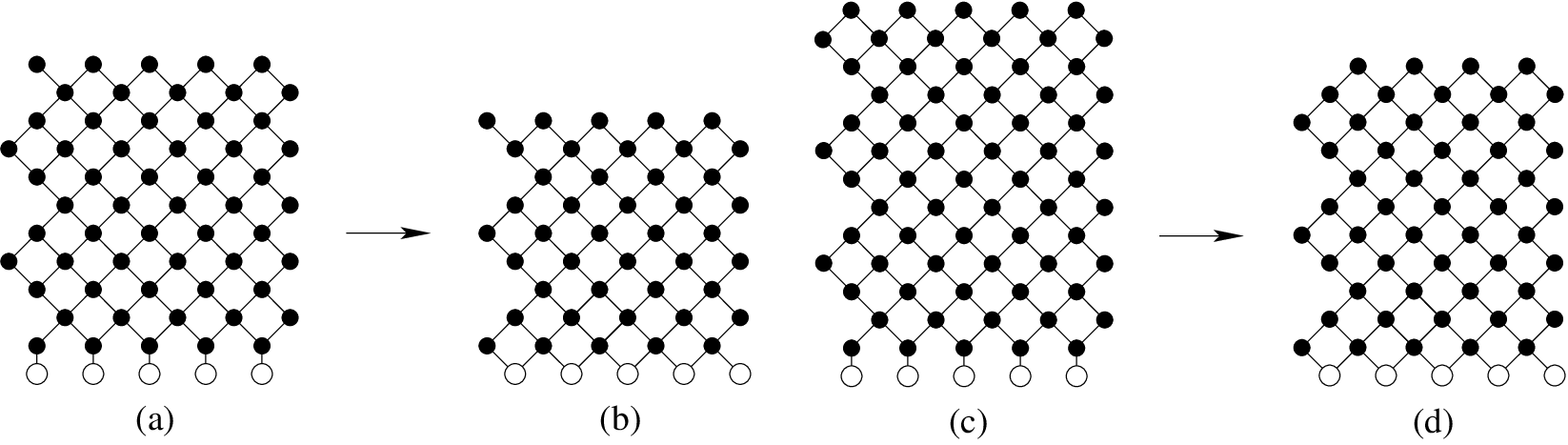}
\caption{Illustrating the transformation in (\ref{T1a}) of Lemma \ref{newT1}.}
\label{symT1}
\end{figure}

\begin{figure}\centering
\includegraphics[width=12cm]{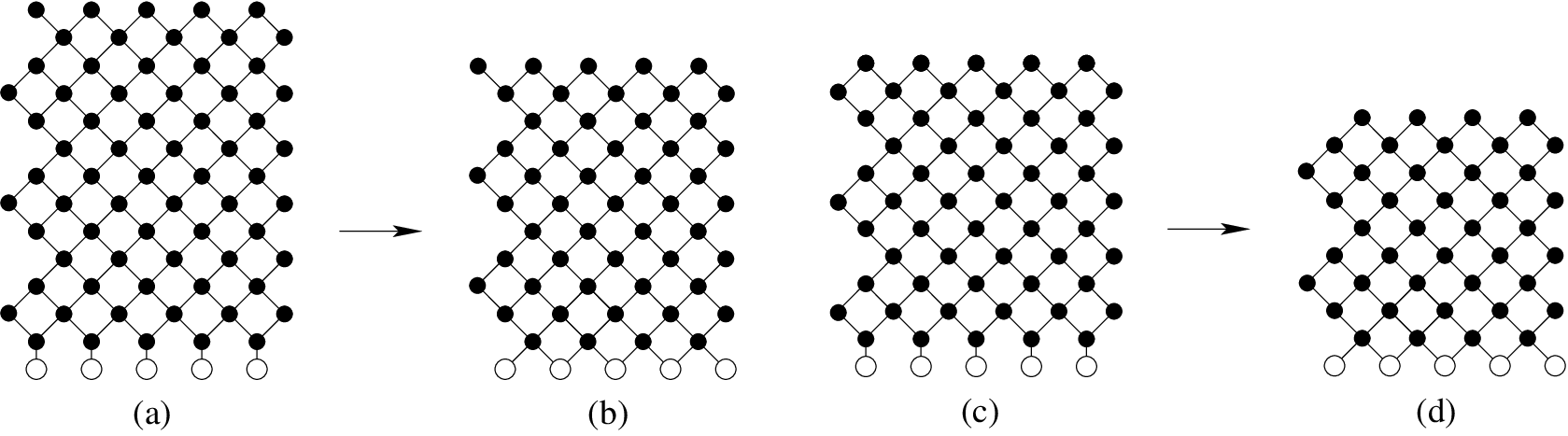}
\caption{Illustrating the transformation in (\ref{T1b}) of Lemma \ref{newT1}.}
\label{symT2}
\end{figure}

\begin{lem}\label{newT1}
\begin{equation}\label{T1a}
\M({}_|AR^{o}_{m,n}\#G)=2^{\lfloor \frac{m}{2}\rfloor}\M(OR^{e}_{m,n}\#G)
\end{equation}
and
\begin{equation}\label{T1b}
\M({}_|AR^{e}_{m,n}\#G)=2^{\lfloor \frac{m+1}{2}\rfloor}\M(OR^{o}_{m,n}\#G),
\end{equation}
where ${}_|AR^{o}_{m,n}$ and ${}_|AR^{e}_{m,n}$ are the graphs  obtained from $AR^{o}_{m,n}$ and $AR^{e}_{m,n}$ by appending $n$ vertical edges to their bottommost vertices, respectively (see Figure \ref{symT1} for examples of the `transformation' in (\ref{T1a}), and Figure \ref{symT2} for examples of the transformation in (\ref{T1b})).
\end{lem}

\begin{figure}\centering
\includegraphics[width=10cm]{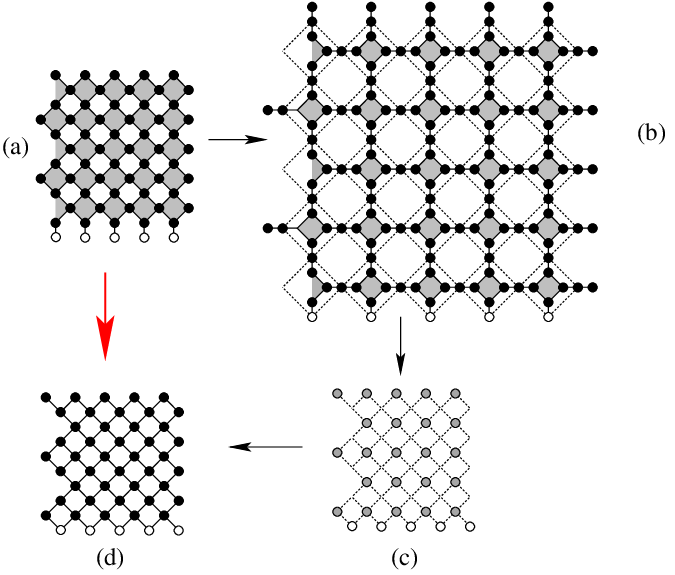}
\caption{Illustration of the proof of Lemma \ref{newT1}.}
\label{proofT1}
\end{figure}

\begin{proof}
We only prove here the transformation in (\ref{T1a}), based on Figure \ref{proofT1},  for $m=n=5$, as the transformation in (\ref{T1b}) can be obtained in the same way.

First, we apply the Vertex-splitting Lemma  (Lemma \ref{VS}) to all vertices of  ${}_|AR^{o}_{m,n}\#G$ that are incident to a shaded diamond or a partial diamond as in Figure \ref{proofT1}(a). We get the graph $G_1$ on Figure \ref{proofT1}(b). Next, we apply the Spider Lemma (Lemma \ref{spider}) around $mn$ shaded diamonds and partial diamonds (the dotted edges have weight $\frac12$), and remove all edges incident to a vertex of degree 1, which are forced. We obtain a weighted graph $G_2$ obtained from $OR^{e}_{m,n}\#G$ by assigning to each edge of $OR^{e}_{m,n}$ a weight $\frac12$. Finally, we get back the graph  $OR^{e}_{m,n}\#G$ by applying the Star Lemma (Lemma \ref{star}) with factor $t=2$ at $mn-\lfloor \frac{m}{2}\rfloor$ shaded vertices as in Figure \ref{proofT1}(c). By Lemmas \ref{VS}, \ref{star}, and \ref{spider}, we have
\begin{align}
\M({}_|AR^{o}_{m,n}\#G)=\M(G_1)=2^{mn}\M(G_2)=2^{mn}2^{-(mn-\lfloor \frac{m}{2}\rfloor)}\M(OR^{e}_{m,n}\#G),
\end{align}
which implies  (\ref{T1a}).
\end{proof}

\begin{figure}\centering
\includegraphics[width=12cm]{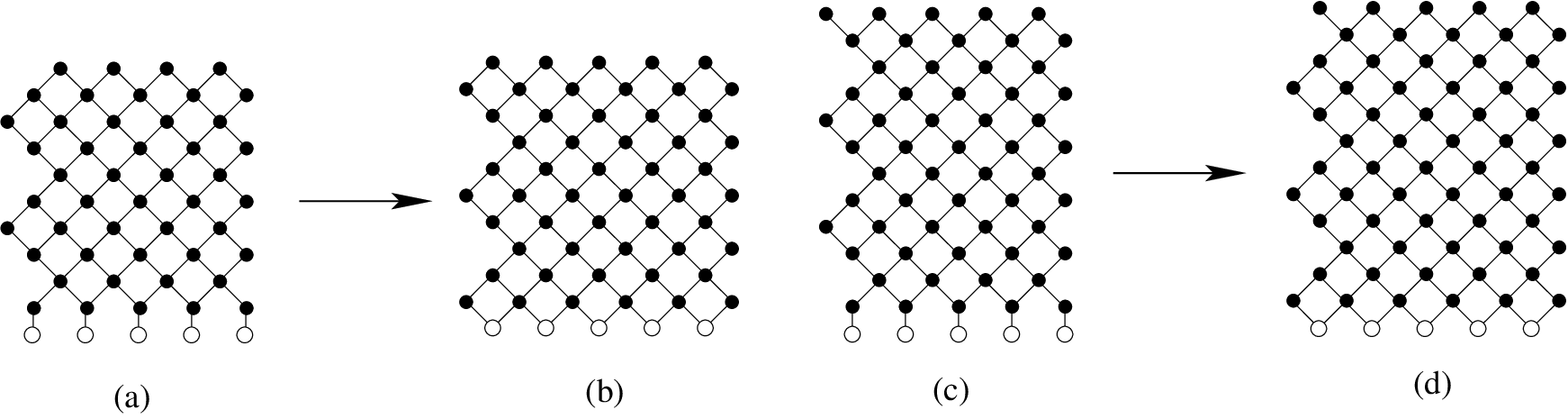}
\caption{Illustrating the transformation in (\ref{T2a}) of Lemma \ref{newT2}.}
\label{symT6}
\end{figure}

\begin{figure}\centering
\includegraphics[width=12cm]{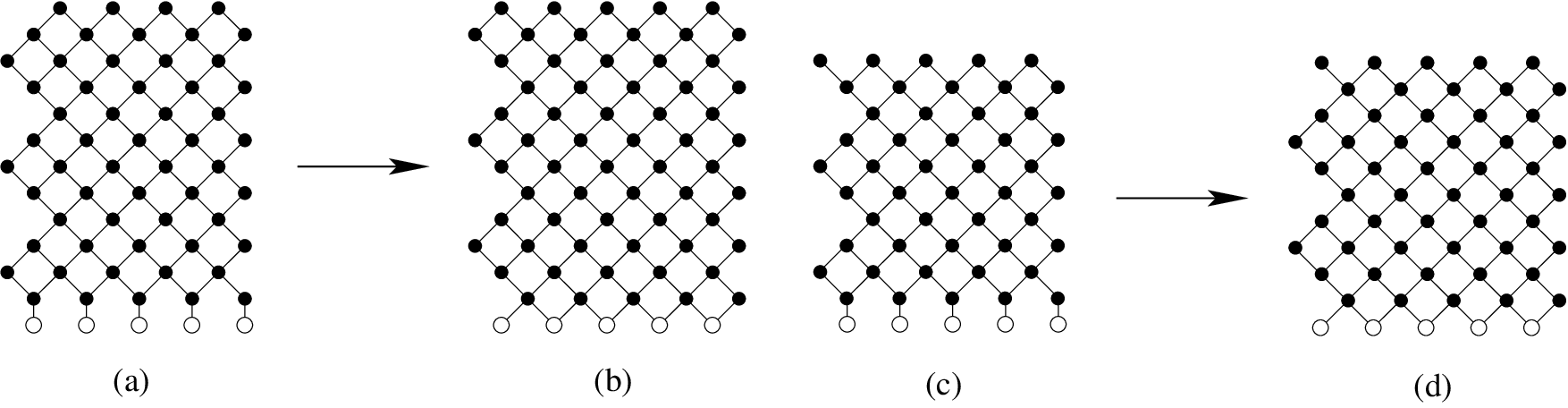}
\caption{Illustrating the transformation in (\ref{T2b}) of Lemma \ref{newT2}.}
\label{symT5}
\end{figure}

By applying the transformations in Lemma \ref{newT1} (in reverse), and then the Vertex-splitting Lemma, one can get the following transformations.

\begin{lem}\label{newT2}

\begin{equation}\label{T2a}
\M({}_|OR^{o}_{m,n}\#G)=2^{-\lfloor \frac{m+1}{2}\rfloor}\M(AR^{e}_{m,n}\#G)
\end{equation}
and
\begin{equation}\label{T2b}
\M({}_|OR^{e}_{m,n}\#G)=2^{-\lfloor \frac{m}{2}\rfloor}\M(AR^{o}_{m,n}\#G),
\end{equation}
where ${}_|OR^{o}_{m,n}$ and ${}_|OR^{e}_{m,n}$ are the graphs obtained from $OR^{o}_{m,n}$ and $OR^{e}_{m,n}$ by appending $n$ vertical edges to their bottommost vertices, respectively (Figure \ref{symT6} shows the transformation in (\ref{T2a}), and Figure \ref{symT5} illustrates the transformation in (\ref{T2b})).
\end{lem}

\begin{figure}\centering
\includegraphics[width=12cm]{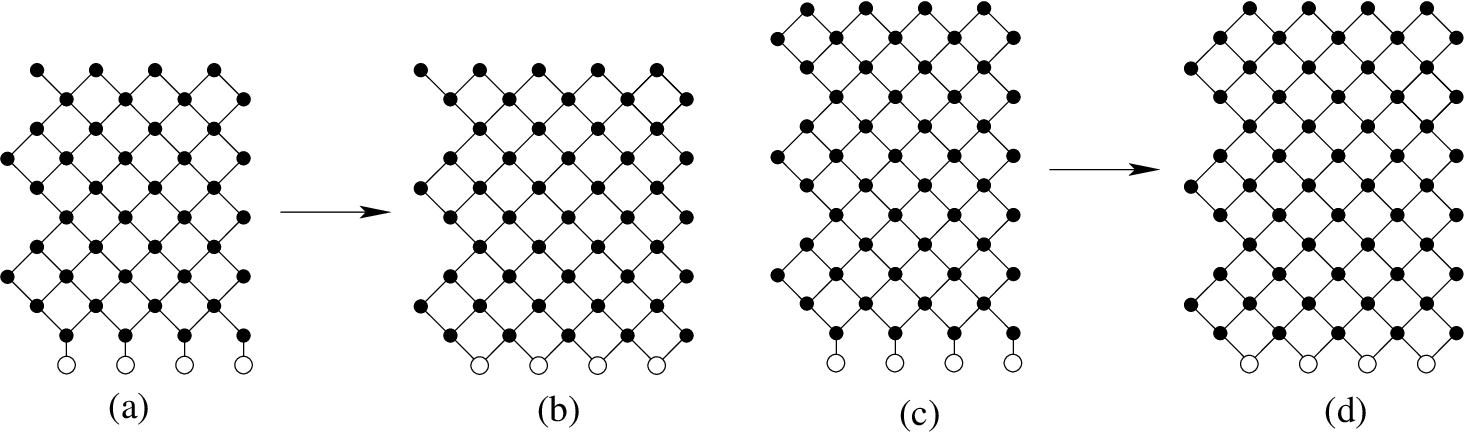}
\caption{Illustrating the transformation in (\ref{T3a}) of Lemma \ref{newT3}.}
\label{symT4}
\end{figure}

\begin{figure}\centering
\includegraphics[width=12cm]{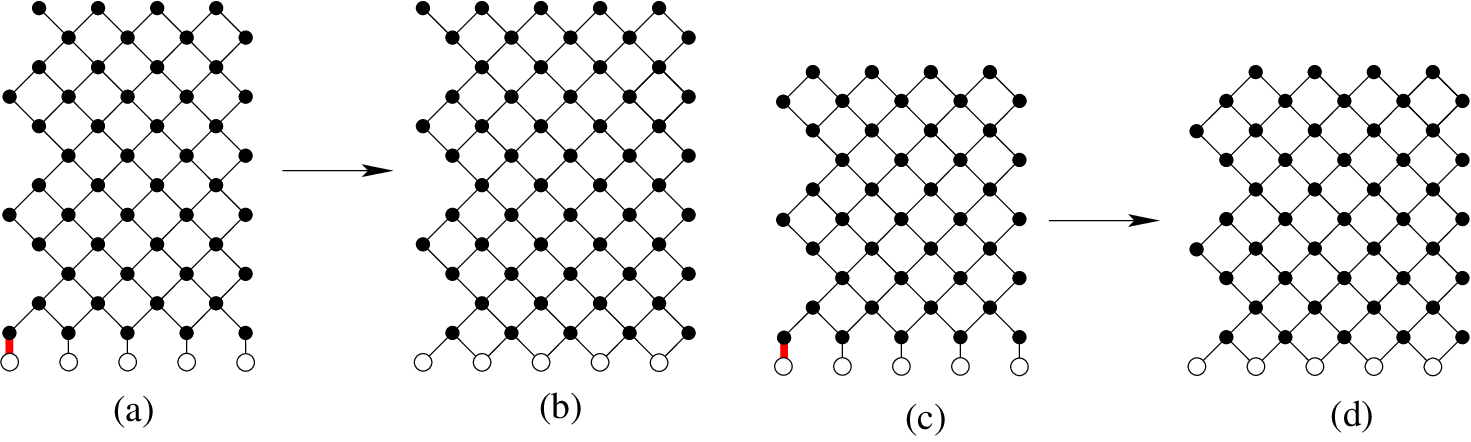}
\caption{Illustrating the transformation in (\ref{T3b}) of Lemma \ref{newT3}. The red bold edges at the lower-left corners of the graphs (a) and (c) are weighted by $\frac12$.}
\label{symT3}
\end{figure}

\newpage

By using Ciucu's Lemma \ref{4cycle} together with Lemmas \ref{VS}--\ref{spider}, one gets the following lemma.

\begin{lem}\label{newT3}
 \begin{equation}\label{T3a}
\M({}_|AR^{o}_{m-\frac12,n}\#G)=2^{-\lfloor \frac{m}{2}\rfloor}\M(AR^{o}_{m,n+\frac12}\#G)
\end{equation}
and
\begin{equation}\label{T3b}
\M({}_|AR^{e}_{m-\frac12,n-1}\#G)=2^{-\lfloor \frac{m+1}{2}\rfloor}\M(AR^{e}_{m,n-\frac12}\#G),
\end{equation}
where ${}_|AR^{o}_{m-\frac12,n}$ is the graph obtained from $AR^{o}_{m-\frac12,n}$ by appending $n$ vertical edges to its bottommost vertices; and where ${}_|AR^{e}_{m-\frac12,n-1}$ is the graph obtained from $AR^{e}_{m-\frac12,n-1}$ by appending $n$ vertical edges to its bottommost vertices, the leftmost vertical edge is weighted by $\frac12$ (the transformation in (\ref{T3a}) is shown in Figure \ref{symT4}, and the transformation in (\ref{T3b}) is illustrated in Figure \ref{symT3}).
\end{lem}

\begin{figure}\centering
\includegraphics[width=10cm]{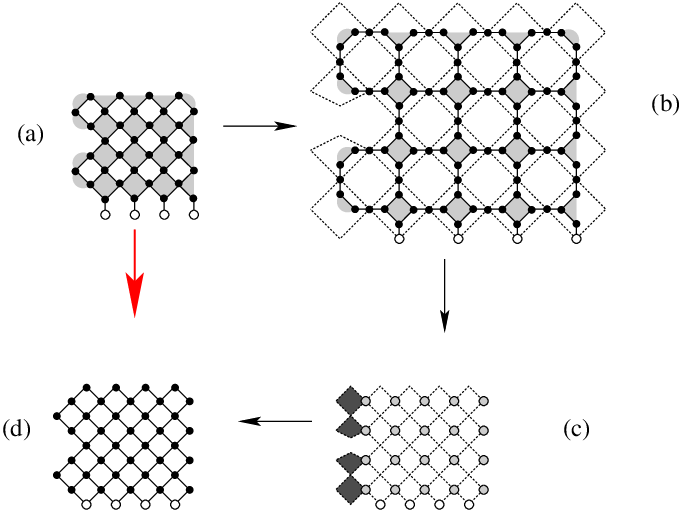}
\caption{Illustrating the proof of (\ref{T3a}) in Lemma \ref{newT3}.}
\label{proofT3}
\end{figure}

\begin{figure}\centering
\includegraphics[width=10cm]{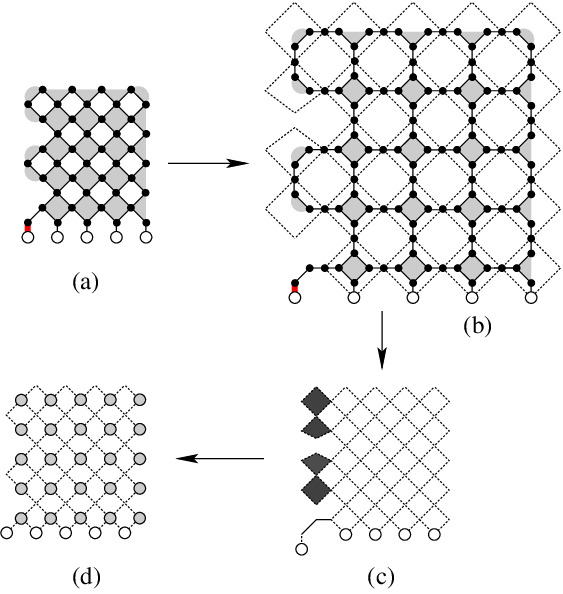}
\caption{Illustrating the proof of (\ref{T3b}) in Lemma \ref{newT3}.}
\label{proofT3b}
\end{figure}

\begin{proof}
We only need to prove (\ref{T3a}) for even $m$, and the case of odd $m$ follows from the even case by removing the southeast-to-northwest forced edges on the top of ${}_|AR^{o}_{m-\frac12,n}$ and $AR^{o}_{m,n+\frac12}$.

Our proof is illustrated in Figure \ref{proofT3}, for $m=4$ and $n=4$.

 First, apply  the Vertex-splitting Lemma to the vertices in ${}_|AR^{o}_{m-\frac12,n}\#G$ that are incident to a shaded diamond or a partial diamond (see Figures \ref{proofT3}(a) and (b)). Second, apply suitable replacement in the Spider Lemma  around $mn$ shaded diamonds and partial diamonds. Third, apply Lemma \ref{4cycle} to remove $\frac{m}{2}$ 7-vertex subgraphs consisting of two shaded 4-cycles (see Figure \ref{proofT3}(c); the dotted edges are weighted by $\frac12$). Finally, apply the Star Lemma with factor $t=2$ to all $m(n+1)$ shaded vertices as in Figure \ref{proofT3}(c). The resulting graph is exactly $AR^{o}_{m,n+\frac12}\#G$. Keeping track the weight factors in the above transformations, we obtain the following equality
\begin{equation}
\M({}_|AR^{o}_{m-\frac12,n}\#G)=2^{m(n+1)}2^{- \frac{m}{2}}2^{-m(n+1)}\M(AR^{o}_{m,n+\frac12}\#G),
\end{equation}
which implies (\ref{T3a}).

Next, we show the proof of (\ref{T3b}) for odd $m$, the case of even $m$ follows from the odd case by removing  southeast-to-northwest forced edges on the top of ${}_|AR^{e}_{m-\frac12,n-1}$ and $AR^{e}_{m,n-\frac12}$. Our proof  is shown in Figure \ref{proofT3b}, for $m=5$ and $n=4$. We apply the Vertex-splitting Lemma to the vertices in ${}_|AR^{e}_{m-\frac12,n}\#G$ incident to a shaded diamond or partial diamond  as in Figures \ref{proofT3b}(a) and (b). Then  apply  the Spider Lemma to $m(n+1)-1$ shaded diamonds and partial diamonds. Next, we apply Lemma \ref{4cycle} to remove $\frac{m-1}{2}$ subgraphs consisting of two shaded 4-cycles (see Figure \ref{proofT3b}(c); the dotted edges have weight $\frac12$), and apply the Vertex-splitting Lemma (in reverse) to eliminate the two solid edges in the resulting graph. Finally, apply the Star Lemma (for the factor $t=2$) to all $m(n+1)$ shaded vertices. This way, we obtain the graph $AR^{e}_{m,n-\frac12}\#G$ on the right-hand side of  (\ref{T3b}).  In summary, we get the following equality:
\begin{equation}
\M({}_|AR^{e}_{m-\frac12,n}\#G)=2^{m(n+1)-1}2^{-\frac{m-1}{2}}2^{-m(n+1)}\M(AR^{e}_{m,n+\frac12}\#G),
\end{equation}
which yields (\ref{T3b}).
\end{proof}

\begin{figure}\centering
\includegraphics[width=12cm]{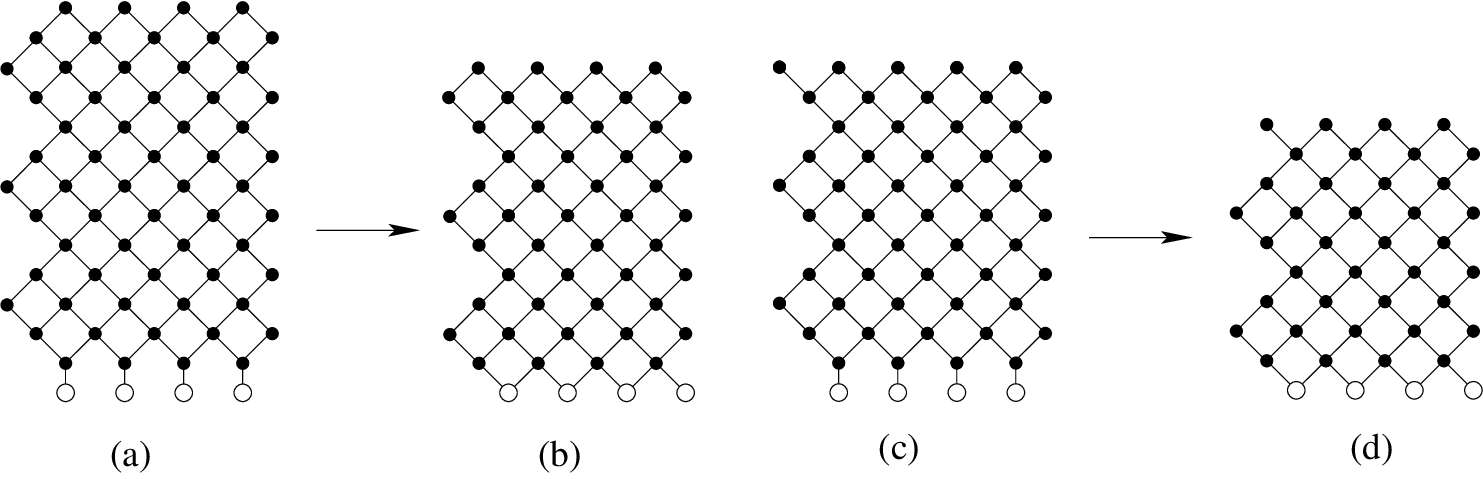}
\caption{Illustrating the transformation in (\ref{T4a}) of Lemma \ref{newT4}}
\label{symT8}
\end{figure}

\begin{figure}\centering
\includegraphics[width=12cm]{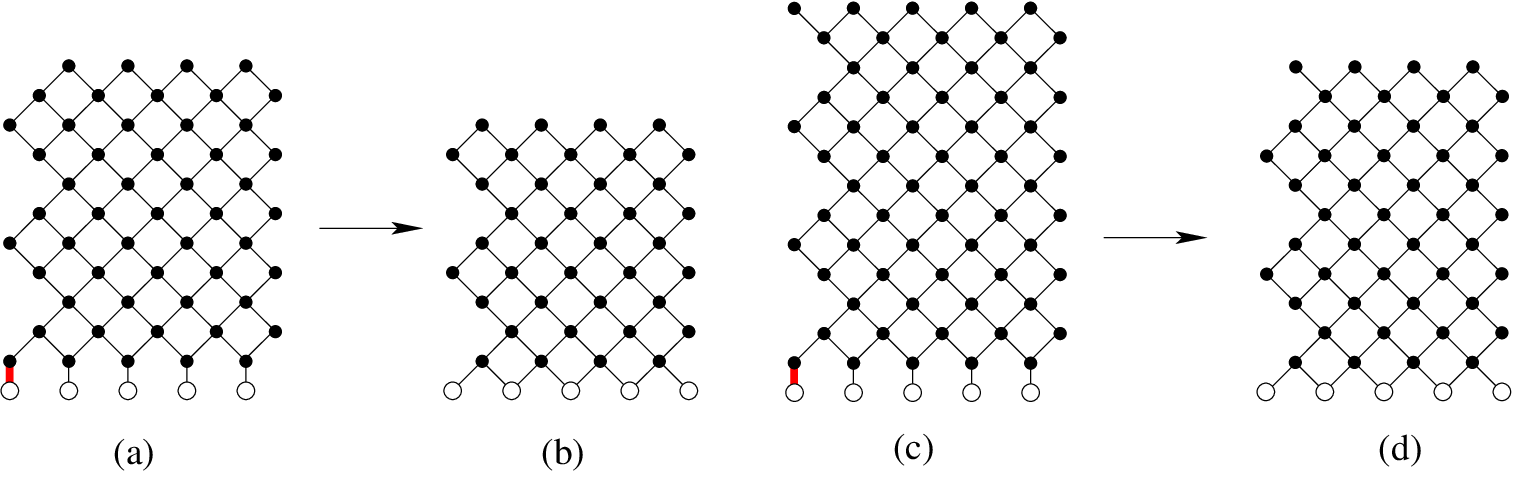}
\caption{Illustrating the transformation in (\ref{T4b}) of Lemma \ref{newT4}. Red edge has weight $\frac12$.}
\label{symT7}
\end{figure}

Similar to Lemma \ref{newT3}, we have the following lemma. The proof of the next lemma is essentially the same as that of Lemma \ref{newT3}, and will be  omitted.

\begin{lem}\label{newT4}

\begin{equation}\label{T4a}
\M({}_|AR^{o}_{m,n+\frac12}\#G)=2^{\lfloor \frac{m}{2}\rfloor}\M(AR^{o}_{m-\frac12,n}\#G),
\end{equation}
and
\begin{equation}\label{T4b}
\M({}_|AR^{e}_{m,n-\frac12}\#G)=2^{\lfloor \frac{m+1}{2}\rfloor-1}\M(AR^{e}_{m-\frac12,n-1}\#G),
\end{equation}
where ${}_|AR^{o}_{m,n+\frac12}$ is the graph obtained from $AR^{o}_{m,n+\frac12}$ by appending $n$ vertical edges to its bottommost vertices; and where ${}_|AR^{e}_{m,n-\frac12}$ is the graph obtained from $AR^{e}_{m,n-\frac12}$ by appending $n$ vertical edges to its bottommost vertices, the leftmost vertical edge is weighted $\frac12$ (the transformation in (\ref{T4a}) is shown in Figure \ref{symT8}, and the transformation in (\ref{T4b}) is illustrated in Figure \ref{symT7}).
\end{lem}

\bigskip

We are now ready to prove Theorem \ref{symthm1}.

\begin{proof}[Proof of Theorem \ref{symthm1}]
We only show in detail the proof for the case when $\alpha$ passes white squares and $w$ is even, as the other cases can be obtained in the same manner.

\medskip

We recall that $\alpha$ is not a drawn-in diagonal, and that $k$ is odd in this case. 

Let $Q$ be any graph with the vertical and horizontal symmetry axes $\ell'$ and $\ell$. We define the \emph{orbit graph}   $\Ob(Q)$ of $Q$ similarly to the proof of Theorem \ref{sym}. In particular, we consider the subgraph $H$ of $Q$ that is induced by the vertices lying on the vertical axis $\ell'$ or staying on the right of $\ell'$. The orbit graph $\Ob(Q)$ of $Q$ is  the graph obtained by identifying two vertices of $H$ on $\ell'$ that have the same distance to the symmetry center $\textbf{O}$, so that the new vertices in $\Ob(Q)$ are on the $\ell$ (i.e. $\Ob(Q)$ also has the horizontal symmetry axis $\ell$). There is always a bijection between the the centrally symmetric perfect matchings of $Q$ and the perfect matchings of its orbit graph $\Ob(Q)$, i.e.
\[\M^*(Q)=\M(\Ob(Q)).\]

\begin{figure}\centering
\includegraphics[width=12cm]{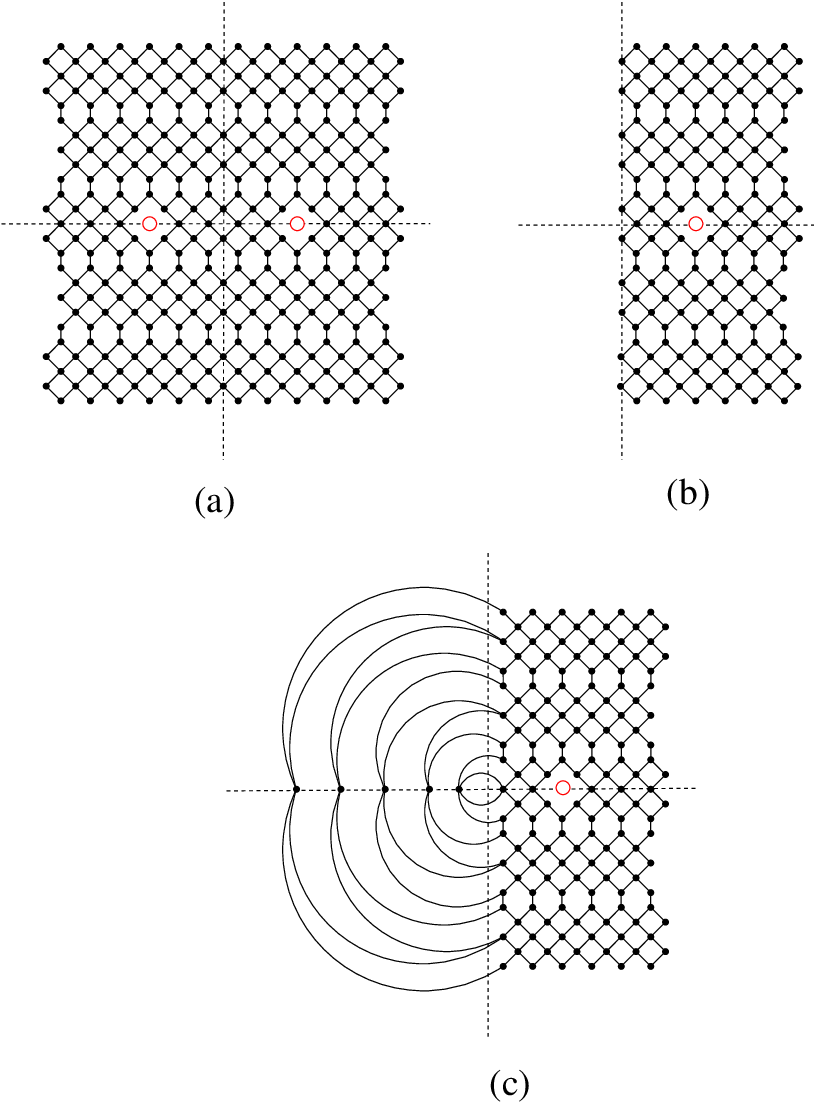}
\caption{The dual graph of the holey Douglas region $\mathcal{D}_{12}(4,4,4,4,4; \{3\})$ and its orbit graph.}
\label{symF1}
\end{figure}

Consider the dual graph $G$ of the region $\mathcal{D}_{a}(\textbf{d};\mathcal{S})$ (rotated $45^\circ$). Its orbit graph $\Ob(G)$ is illustrated in Figure \ref{symF1}. The Factorization Theorem tells us that
\begin{equation}\label{FTeq1}
\M(\Ob(G))=2^{\eta(\Ob(G))}\M(\Ob(G)^+)\M(\Ob(G)^-),
\end{equation}
where $\eta(\Ob(G))$ is half of the number of vertices of $\Ob(G)$ on its horizontal symmetry axis, and where the two component graphs $\Ob(G)^+$ and $\Ob(G)^-$ are illustrated in Figure \ref{symF2}.


\medskip

By Theorem \ref{QARthm}, we only need to show that
\begin{equation}\label{symrefine}
\M^*(\mathcal{D}_{a}(\textbf{d};\mathcal{S}))=\M(\Ob(G)))= 2^{\mathcal{C}-wh-\tau}\M^*(AR_{2h,w}(\mathcal{S})).
\end{equation}

The $k-1$ drawn-in diagonals divide the region $\mathcal{D}_{a}(\textbf{d};\mathcal{S})$ into $k$ parts, called \emph{layers}.
We prove (\ref{symrefine}) by induction on the number of layers $k$  of $\mathcal{D}=\mathcal{D}_{a}(d_1,d_2,\dotsc,d_k;\mathcal{S})$ (recall that $k$ is odd by the symmetry of the Douglas region).

\medskip

If $k=1$, then the dual graph of $\mathcal{D}_{a}(\textbf{d};\mathcal{S})$ is exactly $AR_{2h,w}(\mathcal{S})$, and (\ref{symrefine}) is a trivial identity.  Assuming that (\ref{symrefine}) is true for all  Douglas regions with holes that have less than $k$ layers, $k\geq 3$,
we need to show that (\ref{symrefine}) also holds for any region with $k$ layers $\mathcal{D}=\mathcal{D}_{a}(d_1,d_2,\dotsc,d_k;\mathcal{S})$.

\begin{figure}\centering
\includegraphics[width=12cm]{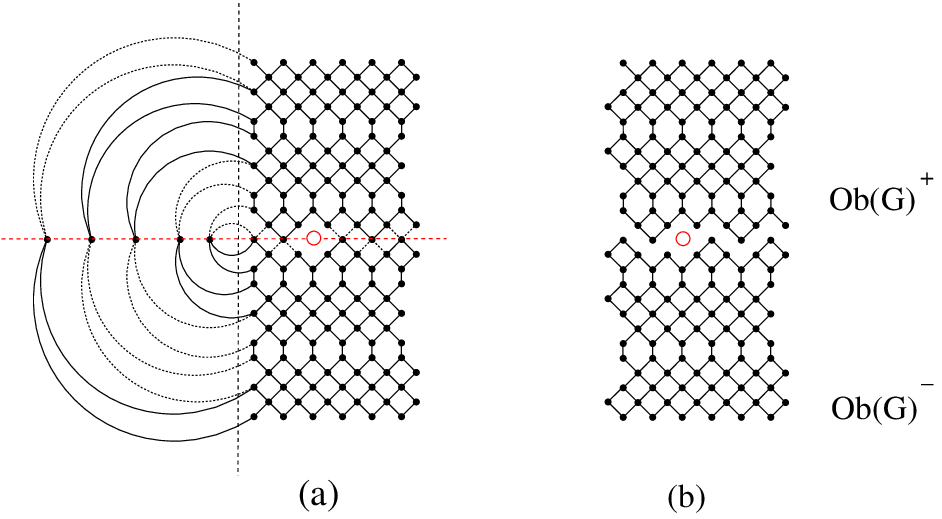}
\caption{Separating the orbit graph of $\mathcal{D}_{12}(4,4,4,4,4; \{3\})$ into two component graphs.}
\label{symF2}
\end{figure}

 There are four cases to distinguish, based on the parities of $d_1$ and $a$.

\medskip

\quad\textit{Case 1. $d_1$ and $a$ are even.}

\medskip

Define a new Douglas region with holes $\mathcal{D}'$ by
\[\mathcal{D}':=\mathcal{D}_{a-1}(d_1+d_2-1,d_3,d_4,\dotsc,d_{k-2}, d_{k-1}+d_k-1; \mathcal{S})  \text{ for $k\geq 5$},\]
and
\[\mathcal{D}':=\mathcal{D}_{a-1}(d_1+d_2+d_3-2; \mathcal{S}) \text{ for $k=3$}.\]
Denote by $G'$ the dual graph of $\mathcal{D}'$. We note that $\mathcal{D}'$ always has $(k-2)$ layers.

The application of the Factorization Theorem to the orbit graph of $G'$ implies
\begin{equation}\label{FTeq2}
\M(\Ob(G'))=2^{\eta(\Ob(G'))}\M(\Ob(G')^+)\M(\Ob(G')^-),
\end{equation}
where $\eta(\Ob(G'))$ is half number of the vertices of $\Ob(G')$ on the symmetry axis (see Figures \ref{symF3}(e) and (f)).

Assume that $d_1/2=2q$. Applying transformation (\ref{T1b}) in Lemma \ref{newT1} to the top part of $\Ob(G)^+$, that corresponds to the first layer of the region $\mathcal{D}$, we get the lower component graph $\Ob(G')^-$ of the orbit graph $\Ob(G')$ of $G'$, and obtain
\begin{equation}
\M(\Ob(G)^+)=2^{q}\M(\Ob(G')^-).
\end{equation}
This process is illustrated in Figures \ref{symF3}(a) and (b).

Similarly, we apply transformation in (\ref{T1a}) in Lemma \ref{newT1} to the bottom part of $\Ob(G)^-$, that corresponds to the bottom layer of $\mathcal{D}$, and get the upper component graph $\Ob(G')^+$ of  of the orbit graph $\Ob(G')$  of $G'$. This implies that
\begin{equation}
\M(\Ob(G)^-)=2^{q}\M(\Ob(G')^+).
\end{equation}
This process is shown in Figures \ref{symF3}(c) and (d).

Multiplying the two equalities above, we get
\begin{equation}\label{FTeq2b}
\M(\Ob(G)^+)\M(\Ob(G)^-)=2^{2q}\M(\Ob(G')^+)\M(\Ob(G')^-).
\end{equation}

Equalities (\ref{FTeq1}), (\ref{FTeq2}), and (\ref{FTeq2b})  now yield
\begin{equation}\label{case1a}
\M(\Ob(G))=2^{2q+\eta(\Ob(G))-\eta(\Ob(G'))}\M(\Ob(G')).
\end{equation}

Since we are assuming that $\alpha$ passes white squares, the number of squares removed from $\alpha$ is $w-2h$. It means that the number vertices of $G$ on $\ell$ is $2h$. Moreover, it is easy to see that the number of vertices of $G$ running along the vertical symmetry axis $\ell'$ is also $2h$. Thus, $\eta(\Ob(G))=h$. Similarly, $\eta(\Ob(G'))=h'$, where $h'$ is the height of $\mathcal{D}'$.

One readily sees that $\mathcal{D}$ and $\mathcal{D}'$ have the same height, so $\eta(\Ob(G))=\eta(\Ob(G'))$ in this case. It means that (\ref{case1a}) can be simplified to
\begin{equation}\label{case1b}
\M(\Ob(G))=2^{2q}\M(\Ob(G')).
\end{equation}

\medskip

Similarly, if $d_1/2=2q+1$, then we can transform the graph $\Ob(G)^+$ into the graph $\Ob(G')^-$ by applying transformation (\ref{T1b}) in Lemma \ref{newT1} to the top part of $\Ob(G)^+$.
This gives us
\begin{equation}
\M(\Ob(G)^+)=2^{q}\M(\Ob(G')^-).
\end{equation}
Next, applying  transformation   (\ref{T1a}) in Lemma \ref{newT1} to the bottom part of $\Ob(G)^-$, we get the graph $\Ob(G')^+$  and
\begin{equation}
\M(\Ob(G)^-)=2^{q+1}\M(\Ob(G')^+).
\end{equation}
Therefore, similar to (\ref{case1b}),  we have the following connection between the numbers of perfect matchings $\Ob(G)$ and $\Ob(G')$:
\begin{equation}\label{case1c}
\M(\Ob(G))=2^{2q+1}\M(\Ob(G')).
\end{equation}

\medskip
\begin{figure}\centering
\includegraphics[width=13cm]{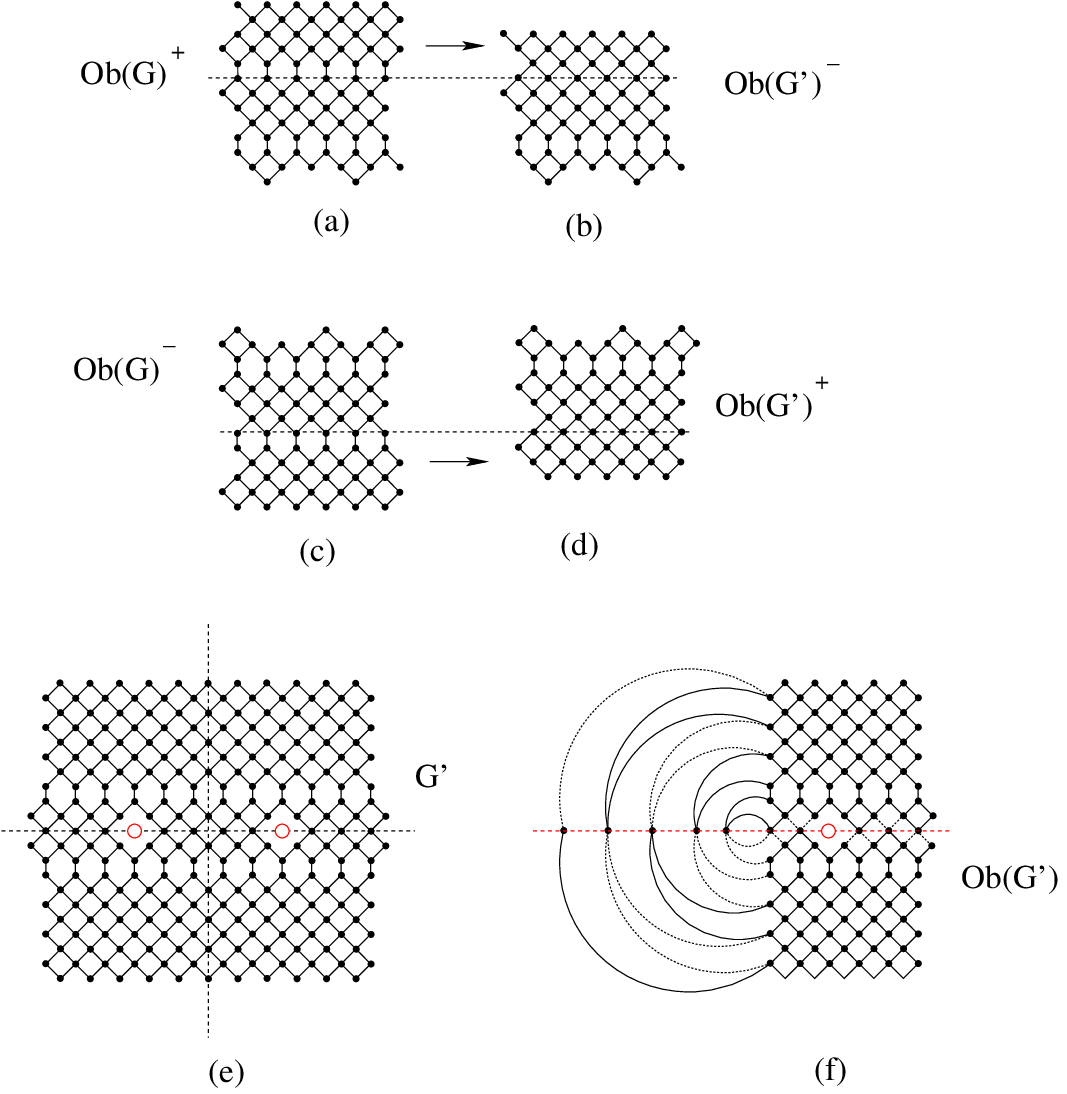}
\caption{Transforming the orbit graph of $\mathcal{D}=\mathcal{D}_{12}(4,4,4,4,4; \{3\})$ into the orbit graph of $\mathcal{D}'=\mathcal{D}_{12}(7,4,7; \{3\})$.}
\label{symF3}
\end{figure}

We can combine (\ref{case1b}) and (\ref{case1c}) into a single identity:
\begin{equation}\label{case1d}
\M(\Ob(G))=2^{d_1/2}\M(\Ob(G')).
\end{equation}

Assume that $\alpha'$ is the axis of $\mathcal{D}'$. Denote by $w', \mathcal{C}',\tau'$ the width, the number of  black regular cells above $\alpha'$, and the defect of $\mathcal{D}'$, respectively.
One readily sees that $w=w'$, $\tau=\tau'$, and
\[\mathcal{C}-\mathcal{C}'=\frac{d_1}{2}(a+1)-\frac{d_1}{2}a=\frac{d_1}{2}.\]

By induction hypothesis for the region $\mathcal{D}'$, we have
\begin{align}
\M(\Ob(G))&=2^{\mathcal{C}'-w'h'-\tau'}\M^*(AR_{2h',w'}(\mathcal{S}))\\
&=2^{\mathcal{C}-d_1/2-wh-\tau}\M^*(AR_{2h,w}(\mathcal{S})),
\end{align}
and  (\ref{symrefine}) follows from (\ref{case1d}).

\medskip

\quad\textit{Case 2. $d_1$ and $a$ are odd.}

\medskip

Define a new Douglas region with holes $\mathcal{D}''$ as
\[\mathcal{D}'':=\mathcal{D}_{a+1}(d_1+d_2+1,d_3,d_4,\dotsc,d_{k-2}, d_{k-1}+d_k+1; \mathcal{S}) \text{ for   $k\geq 5$}, \]
and
\[\mathcal{D}'':=\mathcal{D}_{a+1}(d_1+d_2+d_3+2; \mathcal{S}) \text{ for  $k=3$}. \] Denote by $G''$ its dual graph. We also note that $\mathcal{D}''$ always has $(k-2)$ layers.

Similar to Case 1, we now apply the transformations in Lemma \ref{newT2} to the top part of $\Ob(G)^+$ or the bottom part of $\Ob(G)^-$. If $(d_1+1)/2=2q$, then we get
\begin{equation}
\M(\Ob(G))=2^{-2q+\eta(\Ob(G))-\eta(\Ob(G''))}\M(\Ob(G'')),
\end{equation}
and if $(d_1+1)/2=2q+1$, then
\begin{equation}
\M(\Ob(G))=2^{-2q-1+\eta(\Ob(G))-\eta(\Ob(G''))}\M(\Ob(G'')),
\end{equation}
where $\eta(\Ob(G''))$ is half of the number of vertices of $G''$ on its horizontal symmetry axis.

Moreover, $\mathcal{D}$ and $\mathcal{D}''$ also have the same height, so $\eta(\Ob(G))=\eta(\Ob(G''))$. Thus, we always have in this case
\begin{equation}\label{case2d}
\M(\Ob(G))=2^{-(d_1+1)/2}\M(\Ob(G'')).
\end{equation}

Denote by $w'',h'',\mathcal{C}'',\tau''$ the width, the height, the number of black regular cells in the upper part, and the defect of $\mathcal{D}''$, respectively.
We also have $h=h''$, $w=w''$, $\tau=\tau''$. Moreover,
\[\mathcal{C}-\mathcal{C}''=\frac{d_1+1}{2}(a+1)-\frac{d_1+1}{2}(a+2)=-\frac{d_1+1}{2}.\]

Thus (\ref{symrefine}) follows from (\ref{case2d}) and the induction hypothesis for $\mathcal{D}''$.

\medskip

\quad\textit{Case 3. $d_1$ is odd and $a$ is even.}

\medskip

We use the same transforming process as in Case 2 by using suitable transformations in Lemma \ref{newT3} to the top part of $\Ob(G)^+$ or the bottom part of $\Ob(G)^-$. This gives us
\begin{equation}
\M(\Ob(G))=2^{-(d_1+1)/2}\M(\Ob(G'')),
\end{equation}
where $G''$ is the dual graph of the region $\mathcal{D}''$ defined as in Case 2. Similarly to Case 2, we have (\ref{symrefine}).

\medskip

\quad\textit{Case 4. $d_1$ is even and $a$ is odd.}

\medskip

Apply the same procedure as that in Case 1 by using suitable transformations in Lemma \ref{newT4} to the top part of $\Ob(G)^+$ or the bottom part of $\Ob(G)^-$.

If $d_1/2=2q$, then
\begin{equation}
\M(\Ob(G))=2^{2q-1+\eta(\Ob(G))-\eta(\Ob(G'))}\M(\Ob(G'));
\end{equation}
and if $d_1/2=2q+1$, then
\begin{equation}
\M(\Ob(G))=2^{2q+\eta(\Ob(G))-\eta(\Ob(G'))}\M(\Ob(G')),
\end{equation}
where $G'$ is the dual graph of the region $\mathcal{D}'$ defined as in Case 1.
It means that we always have
\begin{equation}\label{case4d}
\M(\Ob(G))=2^{d_1/2-1}\M(\Ob(G')).
\end{equation}

Similar to Case 1, we have $w'=w$, $h=h'$, and $\mathcal{C}-\mathcal{C}'=d_1/2$. However, in this case $\tau'=\tau-1$, since we have a singular row staying right below the first layer of $\mathcal{D}$, but it does not appear in $\mathcal{D}'$. Thus, by the induction hypothesis for $\mathcal{D}'$ and (\ref{case4d}), we have
\begin{align}
\M(\Ob(G))&=2^{d_1/2-1}2^{\mathcal{C}'-h'w'-\tau'}\M^*(AR_{2h',w'}(\mathcal{S}))\\
&=2^{d_1/2-1}2^{\mathcal{C}-d_1/2-hw-(\tau-1)}\M^*(AR_{2h,w}(\mathcal{S})).
\end{align}
Then  (\ref{symrefine}) follows. This finishes our proof.
\end{proof}

\section{Symmetric tilings of quasi-hexagons}\label{Mainproof3}

In this section, we use our transformations in Lemmas \ref{newT1}--\ref{newT4} to prove Theorem \ref{hexsym}.

\begin{proof}[Proof of Theorem \ref{hexsym}]
There are two cases to distinguish based on the color of the up-pointing triangles running along the axis of $H_{a}(\textbf{d};\textbf{d})$. We consider first the case when these triangles  are black.

We consider the dual graph $G$ of the region $\mathcal{H}=\mathcal{H}_{a}(\textbf{d};\textbf{d})$ (rotated $45^\circ$) with the horizontal and vertical axes $\ell$ and $\ell'$. Similar to the proof of Theorem \ref{symthm1}, the number of centrally symmetric tilings of $\mathcal{H}$ is equal to the number of centrally symmetric perfect matchings of its dual graph $G$. The latter number in turn equals the number of perfect matchings of the orbit graph $\Ob(G)$ of $G$.

The region $\mathcal{H}$ has $k$ layers above the axis, called the \emph{upper layers}. Next, we prove by induction on the number of upper layers of $\mathcal{H}$ that
\begin{claim}
\begin{equation}\label{thm3eq1}
\M(\Ob(G))=2^{\mathcal{C}-hw-\tau}\M(\Ob(\overline{G})),
\end{equation}
where $\overline{G}$ is the dual graph of the region $\overline{\mathcal{H}}:=\mathcal{H}_{w-1}(2h-1;2h-1)$.
\end{claim}

\begin{proof} If $k=1$, then (\ref{thm3eq1}) is a trivial  identity, since $\mathcal{H}=\overline{\mathcal{H}}$. Assume that (\ref{thm3eq1}) holds for any symmetric quasi-hexagon with less than $k$ ($k\geq 2$) upper layers, we need to show that the equality holds also for any symmetric quasi-hexagon $\mathcal{H}=\mathcal{H}_{a}(d_1,\dotsc,d_k;d_1,\dotsc,d_k)$.

 We use similar arguments to that in the proof of Theorem \ref{symthm1}. In particular, we transform the orbit graph of $\mathcal{H}_{a}(\textbf{d};\textbf{d})$ into the orbit of a symmetric quasi-hexagon that has less layers by using the suitable transformations in Lemmas \ref{newT1}--\ref{newT4}.

We only show in detail the proof for the case when $a$ and $d_1$ are even, as the other cases can be obtained similarly.

Similar to the Theorem 7.1 in \cite{Ciucu}, we notice that we \emph{cannot} apply Ciucu's Factorization Theorem  directly here, since the vertices of the orbit graph $\Ob(G)$ on $\ell$, $\{a_1,b_1,a_2,b_2,\dotsc,a_{\eta(\Ob(G))},b_{\eta(\Ob(G))}\}$, do \emph{not} form a cut set. However, the Lemma 2.1 in \cite{Ciucu} still applies, and it means that all $2^{\eta(G)}$ graphs, that are obtained from $\Ob(G)$ by cutting edges from above or below each of $a_i$'s, have the same number of perfect matchings. We now consider a cutting procedure at the vertices $a_i$'s as follows. First, we color these vertices of $\Ob(G)$ inductively from left to right: color $a_1$ by white, then color the next vertex the same color as its left one if there is not an edge connecting them, otherwise we use the opposite color (see Figure \ref{symmetricfig4}(b)).  Assume that $Q^*$ is the graph obtained from $\Ob(G)$ by cutting above all white $a_i$'s and below all black $a_i$'s.  We will show in the next paragraph that all perfect matchings of $Q^*$ have the white $b_j$'s matched upward, and black $b_j$'s matched downward.

\begin{figure}\centering
\includegraphics[width=13cm]{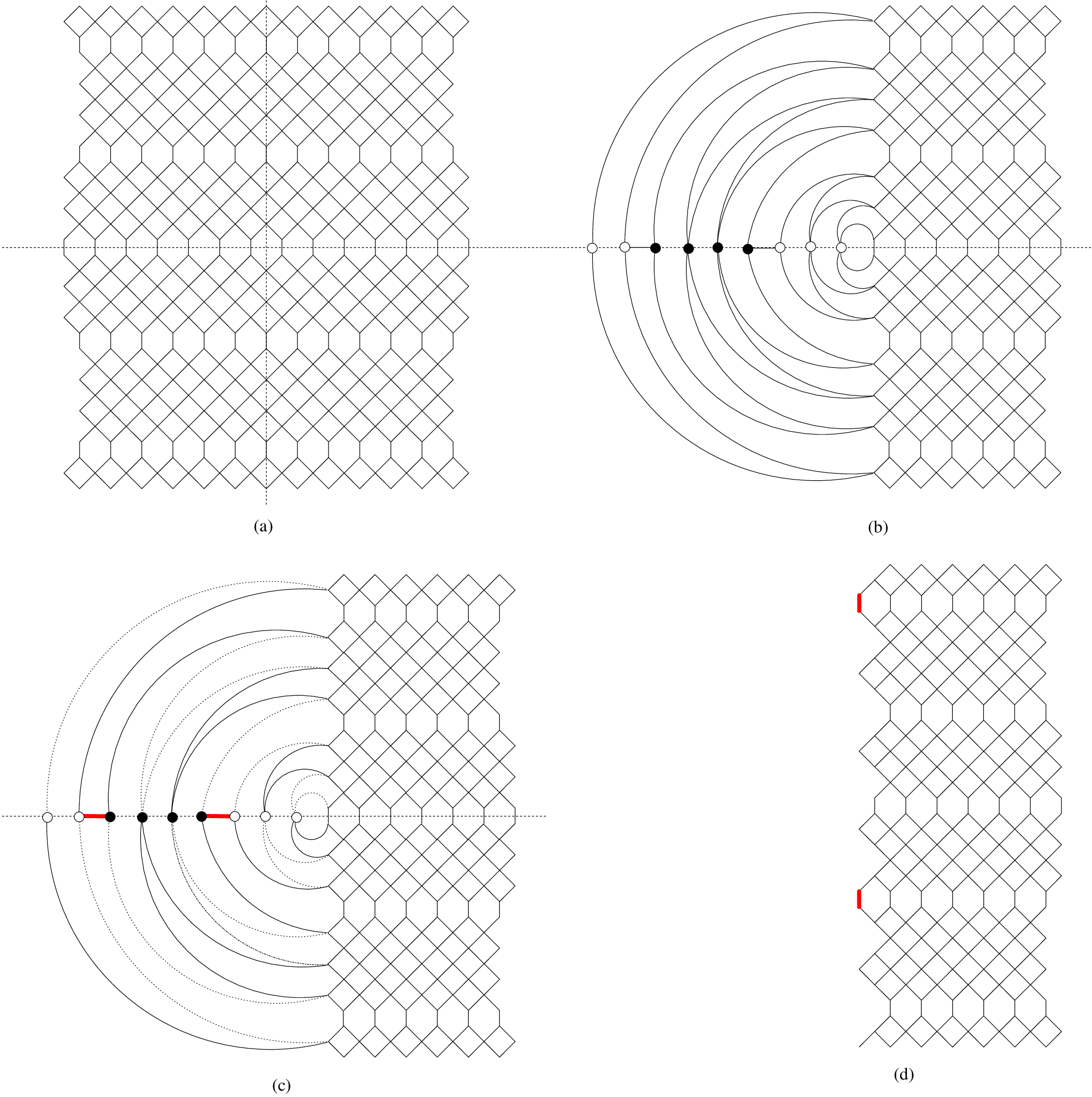}
\caption{Obtaining the graph $Si(G)$ from $G$. The dotted edges indicate the ones cut off. The bold edges have weight $\frac12$.}
\label{symmetricfig4}
\end{figure}

Indeed, we consider the collection $\mathfrak{G}$ of $2^{\eta(G)}$ graphs obtained from $Q^*$ by cutting at all edges incident to $b_j$'s from above or below. The matching set of $Q^*$ is in bijection to disjoint union of matching sets of the members in $\mathfrak{G}$. Recall that if a bipartite graph admits a perfect matching, then its vertex classes must have the same size. All members $Q$ of $\mathfrak{G}$ are bipartite graphs, and it easy to check that its two vertex classes have the same size only if $Q$ is obtained from cutting below all white $b_j$'s and above all black $b_j$'s.

Now, denote by $Si(G)$ the graph obtained from the orbit graph $\Ob(G)$ of $G$ by cutting above all white $a_i$'s and black $b_j$'s, and below all black $a_i$'s and white $b_j$'s (illustrated in in Figure \ref{symmetricfig4}(c)). Moreover, $Si(G)$ can be deformed into a weighted subgraph of $G$ as in Figure \ref{symmetricfig4}(c). Then we get
\begin{equation}
\M(\Ob(G))=2^{\eta(G)}\M(Si(G)).
\end{equation}
 Applying the transformations in Lemma \ref{newT1} to the top and bottom parts of $Si(G)$, that correspond to the top and bottom layers of $\mathcal{H}$,  we get the graph isomorphic to $Si(G')$, where $G'$ is the dual graph of  the quasi-hexagon $\mathcal{H}'$ defined by
\[\mathcal{H}':=\mathcal{H}_{a-1}(d_1+d_2-1,d_3,\dotsc,d_k;\ d_1+d_2-1,d_3,\dotsc,d_k),\]
and where $Si(G')$ is the graph obtained from the orbit graph $\Ob(G')$ of $G'$ by the same cutting procedure as in the case of $G$.
We obtain
\begin{align}\label{FTeq4}
\M(\Ob(G))&=2^{\eta(G)}Si(G)\notag\\
&=2^{\eta(G)}2^{d_1/2}Si(G')\notag\\
&=2^{d_1/2}2^{\eta(G')}Si(G')\notag\\
&=2^{d_1/2}\M(\Ob(G')).
\end{align}
By (\ref{FTeq4}), the induction hypothesis, and explicit calculation of the statistics of the region $\mathcal{H}'$, we obtain
\begin{align}
\M(\Ob(G))&=2^{d_1/2}2^{\mathcal{C}'-h'w'-\tau'}\M(\Ob(\overline{G}))\\
&=2^{d_1/2}2^{\mathcal{C}-d_1/2-hw-\tau}\M(\Ob(\overline{G})),
\end{align}
where $\mathcal{C}',h',w',\tau'$ refer to $\mathcal{H}'$ corresponding to their unprimed counterparts in $\mathcal{H}$.
Then (\ref{thm3eq1}) follows, and this finishes the proof of our claim.
\end{proof}

\bigskip

Consider the dual graph $G''$ of the symmetric quasi-hexagon region
\[\mathcal{H}'':=\mathcal{H}_{w-h}(\textbf{1}^{h};\textbf{1}^{h}),\]
where all $d_1,d_2,\dotsc,d_k$ are $1$, and where $k=h$.
$\mathcal{H}''$ is exactly the semi-regular hexagon of side-lengths $h,w-h,h,h,w-h,h$ on the triangular lattice.
Applying the claim above to the orbit graph $\Ob(G'')$ of $G''$, we have
\begin{equation}\label{thm3eq2}
\M(\Ob(G''))=2^{-h(h-1)/2}\M(\Ob(\overline{G})).
\end{equation}
Thus, by (\ref{thm3eq1}) and (\ref{thm3eq2}), we obtain
\begin{equation}\label{thm3eq3}
\M(\Ob(G))=2^{C-h(2w-h+1)/2-\tau}\M(\Ob(G'')).
\end{equation}
The number of perfect matchings  of $\Ob(G'')$ is given by Ciucu's Theorem 7.1 in \cite{Ciucu6}, and our theorem follows in the case when the triangles right above the axis $\ell$ are black .

\begin{figure}\centering
\includegraphics[width=10cm]{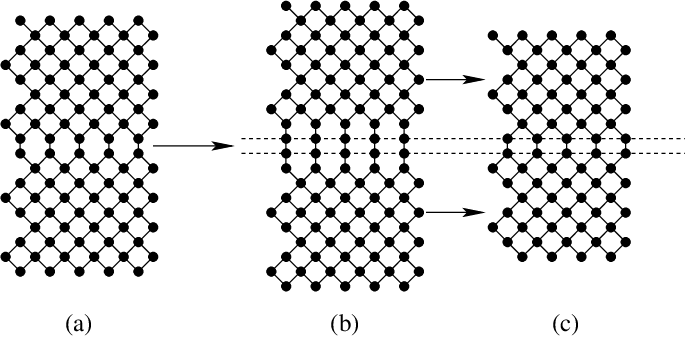}
\caption{Transforming $Si(\overline{\overline{G}})$ into $Si(\overline{G})$.}
\label{symmetricfig5}
\end{figure}

\medskip

Next, we consider the case where the triangles right above $\ell$ are white. Similarly, we can prove by induction on the number of upper layers $k$ that
\begin{align}
\M(\Ob(G))&=2^{C-h(w+1)-\tau}\M(\Ob(\overline{\overline{G}}))\\
&=2^{C-h(w+1)-\tau}2^{\eta(\overline{\overline{G}})}\M(Si(\overline{\overline{G}})),
\end{align}
where $\overline{\overline{G}}$ is the dual graph of the region $\overline{\overline{\mathcal{H}}}:=\mathcal{H}_{w}(2h;2h)$, and $Si(\overline{\overline{G}})$ is the graph obtained from  the orbit graph $\Ob(\overline{\overline{G}})$ of $\overline{\overline{G}}$ by applying the cutting procedure in the previous case.

Next, we apply the Vertex-splitting Lemma to all vertices at the bottom of the upper part of $Si(\overline{\overline{G}})$ (see Figure \ref{symmetricfig5}(b)), and use the suitable transformations in Lemmas \ref{newT1}--\ref{newT4} to transform $Si(\overline{\overline{G}})$ into $Si(\overline{G})$, where $\overline{G}$ is the dual graph of $\mathcal{H}_{w-1}(2h-1;2h-1)$ as defined in the previous case (see Figure \ref{symmetricfig5} (c)). Then this case follows from the case treated above. This finishes the proof of our theorem.
\end{proof}

\section{Concluding remarks}

This paper and its prequels \cite{Tri1, Tri2} have shown the power of the subgraph replacement method in  the enumeration of tilings. The method helps us transform complicated graphs into simple graphs whose matching numbers are known. 

One of the main ingredients of the method is the Spider Lemma (Lemma \ref{spider}). The local transformation in this lemma, that is known as the `\emph{urban renewal}' or `\emph{domino shuffling}', was first found by Greg Kuperberg. James Propp generalized it \cite{Propp2} and used the generalization to prove Stanley's formula for weighted tilings of the Aztec diamond \cite{Propp3}. Douglas later used a variant of the urban renewal to obtain his theorem in \cite{Doug}.

It is worth noticing  that Ciucu developed a useful linear algebraic version of this transformation to obtain tiling formulas for a number of Aztec-diamond-like regions \cite{CiucuP}. See also \cite{Tri8} for a sequel of Ciucu's paper written by the author. 

We refer the reader to e.g. \cite{Ciucu1,Young2,KP,Tri9,Tri10,LM,Young3,Young1} for more applications of the subgraph replacement method.

\end{document}